\documentclass[a4paper, 10pt]{article}

\usepackage{authblk}
\usepackage{amsmath}
\usepackage{amssymb}
\usepackage{amsthm}
\usepackage{bm}
\usepackage{bbm}
\usepackage{enumerate}
\usepackage{a4wide}
\usepackage{tikz}
\usepackage{graphicx}
\usepackage{subcaption}
\usepackage{epstopdf}
\usetikzlibrary{decorations.pathreplacing}
\usepackage{algorithm}
\usepackage{algorithmic}

\newcommand{\pr}{\mathbb{P}}
\newcommand{\Prob}[1]{\pr\left(#1\right)}

\newcommand{\e}{\mathbb{E}}
\newcommand{\Exp}[1]{\e\left[#1\right]}

\newcommand{\plim}{\ensuremath{\stackrel{\pr}{\rightarrow}}}

\newcommand{\ind}[1]{\mathbbm{1}_{\left\{#1\right\}}}

\newcommand{\floor}[1]{\left\lfloor{#1}\right\rfloor}

\newcommand{\bigO}[1]{O\left(#1\right)}

\newcommand\numberthis{\addtocounter{equation}{1}\tag{\theequation}}

\allowdisplaybreaks

\newtheorem{theorem}{Theorem}[section]
\newtheorem{definition}{Definition}[section]
\newtheorem{lemma}[theorem]{Lemma}
\newtheorem{proposition}[theorem]{Proposition}
\newtheorem{corollary}[theorem]{Corollary}
\newtheorem{assumption}[definition]{Assumption}

\title{Generating maximally disassortative graphs with given degree distribution.}
\author[1]{Pim van der Hoorn}
\author[2,3]{Liudmila Ostroumova Prokhorenkova} 
\author[2,3]{Egor Samosvat}
\affil[1]{University of Twente, Enschede, the Netherlands}
\affil[2]{Moscow Institute of Physics and Technology, Moscow, Russia}
\affil[3]{Yandex, Moscow, Russia}

\begin{document}

\maketitle

\begin{abstract}
In this paper we consider the optimization problem of generating graphs with a prescribed degree distribution, such that the correlation between the degrees of connected nodes, as measured by Spearman's rho, is minimal. We provide an algorithm for solving this problem and obtain a complete characterization of the joint degree distribution in these maximally disassortative graphs, in terms of the size-biased degree distribution. As a result we get a lower bound for Spearman's rho on graphs with an arbitrary given degree distribution. We use this lower bound to show that for any fixed tail exponent, there exist scale-free degree sequences with this exponent such that the minimum value of Spearman's rho for all graphs with such degree sequences is arbitrary close to zero. This implies that specifying only the tail behavior of the degree distribution, as is often done in the analysis of complex networks, gives no guarantees for the minimum value of Spearman's rho.
\end{abstract}


\section{Introduction}

An important second order characteristic of the topology of a graph, introduced in \cite{Newman2002}, is the correlation between the degrees at both sides of a randomly sampled edge, also called degree-degree correlation or degree assortativity. A graph is called assortative, or is said to have assortative mixing, if this correlation is positive and disassortive if it is negative. In assortative graphs, nodes of a certain degree have a preference to connect to nodes of similar degree, while in a disassortative graph the opposite is true, for instance, nodes of small degrees connect to nodes with large degrees. When the degrees of connected nodes are uncorrelated the graph is said to have neutral mixing. 

Recently, the problem of generating graphs with a given joint degree structure has been investigated. In \cite{Bassler2015a} and \cite{Stanton2012} algorithms are introduced for constructing and sampling graphs with a given joint degree matrix $J$, where an entry $J_{k \ell}$ denotes the number of edges between nodes of degrees $k$ and $\ell$. An algorithm for generating random graphs whose joint degree distribution converges to a given limiting distribution is given in \cite{Deprez2015} and \cite{Hurd2015} under the assumption that the degrees are uniformly bounded in the size of the graph. 

A different branch of research is concerned with generating graphs that have extreme degree-degree correlation structure, either maximally assortative or disassortative, and analyzing structural properties of such graphs. One algorithm that is often used for this is the so-called edge swap algorithm \cite{Kannan1999, Maslov2002, Xulvi-Brunet2004}. In the context of degree-degree correlations, this algorithm starts from an initial graph, with a prescribed degree sequence, and in each step two edges are sampled and switched based on some rule, in order to obtain a maximally (dis)assortative graph. In \cite{Menche2010} this algorithm is used to obtain scaling results for  Pearson's correlation coefficient, as introduced in \cite{Newman2002} on maximally (dis)assortative graphs where the degrees follow a scale-free distribution. The results from \cite{Menche2010} are extended in \cite{Yang2016}, where a lower bound for Pearson's correlation coefficient is established in scale-free graphs. 

One of the problems with the current analysis of graphs with extreme degree-degree correlation structure is the use of Pearson's correlation coefficient as a measure for assortativity, since this measure has been shown to be size-dependent when the degree distribution has infinite variance \cite{Hofstad2014, Hoorn2014a}. In these papers new, rank-based, correlation measures are introduced and it is shown that these measures converge to a proper limit, determined by the joint degree distribution, under very standard assumptions, see \cite{Hofstad2014,Hoorn2014}. Therefore, in this paper, we follow their suggestion and use a rank correlation measure related to Spearman's rho.

We introduce a greedy algorithm for generating graphs with a given degree distribution that are maximally disassortative, with respect to the rank correlation measure Spearman's rho. The algorithm gives insights into the joint degree structure of these graphs. Using these insights we are able to characterize the limiting joint degree distribution of maximally disassortative graphs, in terms of the size-biased degree distribution. Moreover, due to use of a general framework describing the convergence of the empirical distributions, we are able to characterize the speed of the convergence. 

An important consequence of the joint degree structure of maximally disassortative graphs is that the tail of the distribution does not affect the minimum value of Spearman's rho. Moreover, we are able to construct regularly varying distributions with a prescribed exponent, such that Spearman's rho on any graph with this degree distribution is bounded from below by a value that is arbitrary close to zero. 

We complement our theoretical results with simulations that show the concentration of Spearman's rho for graphs generated by our algorithm and illustrate how this measure is influenced by the shape of the size-biased degree distribution. We observe that the minimal value Spearman's rho becomes larger when more mass is placed in the head of the degree distribution, while increasing the mass in the tail of the distribution decreases this value.

\section{Notations and results}

We will start by introducing some notation and summarizing our main results.

\subsection{Graphs and Degree sequences}

Given a degree sequence ${\bf D}_n = \{D_1, D_2, \dots, D_n\}$ we define $L_n = \sum_{i = 1}^n D_i$. That is, $L_n$ is the sum of the degrees and hence \emph{twice} the number of edges in a graph with degree sequence ${\bf D}_n$. We further define the empirical and sized-biased degree distributions by, respectively,
\begin{align}
	f_n(k) &= \frac{1}{n} \sum_{i = 1}^n \ind{D_i = k}, \\
    f_n^\ast(k) &= \frac{1}{L_n} \sum_{i = 1}^n k \ind{D_i = k},
\end{align}
and let $F_n$ and $F_n^\ast$ be the corresponding cumulative distribution functions. 

We will assume that the empirical distributions $f_n$ and $f_n^\ast$ converge to certain limiting distributions $f$ and $f^\ast$ as follows. 

\begin{assumption}\label{asmp:regularity_distributions}
Let $f$ and $f^\ast$ be probability mass functions on the non-negative integers such that 
\begin{equation}\label{eq:finite_mean_plus_condition}
	\sum_{k = 0}^{\infty} k^{1 + \eta} f(k) < \infty
\end{equation}
for some $\eta > 0$ and if we define, for some $\varepsilon > 0$,
\[
	\Omega_n = \left\{\max\left\{\sum_{k =0}^\infty \left|\sum_{t = 0}^k f_n(t) - f(t)\right|, 
    \quad \sum_{k = 0}^\infty \left|f_n^\ast(k) - f^\ast(k)\right| \right\} 
    \le n^{-\varepsilon} \right\},
\]
then
\[
	\Prob{\Omega_n} \plim 1 \quad \text{as } n \to \infty.
\]
\end{assumption}

We will denote by $F$ and $F^\ast$ the cumulative distributions of $f$ and $f^\ast$, respectively. Since we assume that the event $\Omega_n$ occurs, asymptotically, with probability one, we will often use the probability of its complement $\Omega_n^c$ to describe the speed of convergence in our results. In addition, for simplicity of notation, we will assume throughout this paper that $f(k), f^\ast(k) > 0$ for all $k \ge 0$. Our results extend in a straightforward manner to other cases, by considering only all $k$ for which $f(k), f^\ast(k) > 0$.

To give some explanation regarding Assumption \ref{asmp:regularity_distributions} we remark that the first expression in the maximum of the event $\Omega_n$ is related to the Kantorovich-Rubinstein distance or, equivalently, the Wasserstein metric of order one between the distributions $F_n$ and $F$. Convergence in this metric is equivalent to weak convergence as well as convergence of the first absolute moments, see \cite{Villani2008} for more details. Hence, assumption~\ref{asmp:regularity_distributions}  describes the joint convergence of $f_n$ to $f$ and $f_n^\ast$ to $f^\ast$ in the Kantorovich-Rubinstein distance and the $1$-norm, respectively. We used different metrics for the convergence of $f_n$ and $f_n^\ast$, since the Wasserstein metric is only a true distance when the distributions have finite first absolute moment. We are not assuming that the distribution $f^\ast$ has finite first absolute moment since we want to consider graphs whose degree distributions have infinite second moment, which implies that the size-biased degree distribution has infinite mean. 

In order to state our results we will use the following definition
\begin{definition}
Let ${\bf D}_n$ be a degree sequence. We say that ${\bf D}_n \in \mathcal{D}_{\eta, \, \varepsilon}(f,f^\ast)$ if and only if ${\bf D}_n$ satisfies Assumption \ref{asmp:regularity_distributions} with density functions $f$ and $f^\ast$ and $\eta, \varepsilon > 0$. For a graph $G_n$ with a degree sequence ${\bf D}_n$, we will write $G_n \in \mathcal{G}_{\eta, \, \varepsilon}(f, f^\ast)$ if ${\bf D}_n \in \mathcal{D}_{\eta, \, \varepsilon}(f, f^\ast)$.
\end{definition}

\subsection{Spearman's rho on graphs}

For an integer valued random variable $X$, we denote its cummulative distribution function by $F_X$ and define
\begin{equation}\label{eq:definition_mathcal_F}
	\mathcal{F}_X(k) = F_X(k) + F_X(k - 1), \quad \text{for all } k \in \mathbb{Z}.
\end{equation}
Now, let $X$ and $Y$ be two integer valued random variables, then Spearman's rho is defined as \cite{Mesfioui2005}
\begin{equation}\label{eq:spearmans_rho_theoretical}
	\rho(X,Y) = 3\Exp{\mathcal{F}_X(X)\mathcal{F}_Y(Y)} - 3.
\end{equation}

For the definition of Spearman's rho on graphs it is convenient to consider directed edges. To make this work on undirected graphs we replace each edge $i - j$ by two edges, $i \to j$ and $j \to i$. We refer to this graph as the bi-directed version of the original graph. 
Although the graph on which Spearman's rho is computed is directed, we will not distinguish between this 
and the original undirected graph $G_n$. That is, we will write $i \to j \in G_n$ to mean that
$i \to j$ is present in the bi-directed version of $G_n$, which is equivalent to stating that 
$i - j \in G_n$. We recall that $L_n$ denotes the 
sum over all degrees, so that $L_n$ is \emph{twice} the number of undirected edges and equal
to the corresponding number of directed edges in $G_n$.

Next we will consider Spearman's rho with uniform ranking, as described in \cite{Hofstad2014} and \cite{Hoorn2014a}. That is, we take $({\bf U}_{i \to j}, {\bf W}_{i \to j})$ to be a vector of independent uniform random variables 
$U_{i \to j}$ and $W_{i \to j}$ on $(0, 1)$, for each edge $i \to j \in G$, and define the ranking functions $R_\ast(i \to j)$ and $R^\ast(i \to j)$ by
\begin{align*}
	R_\ast(i \to j) &= \sum_{s \to t \in G} \hspace{-3pt} 
    	\ind{D_s + U_{s \to t} \ge D_i + U_{i \to j}}, \\
    R^\ast(i \to j) &= \sum_{s \to t \in G} \hspace{-3pt} 
    	\ind{D_t + W_{s \to t} \ge D_j + W_{i \to j}},
\end{align*}
where we let $\sum_{i \to j \in G}$ denote the sum over all edges $i \to j$ in the graph $G$.
With these definitions, Spearman's rho is defined as, see \cite{Hofstad2014, Hoorn2014a},
\begin{equation}\label{eq:spearmans_rho_full}
	\rho(G_n) = \frac{12 \sum_{i \to j \in G} R_\ast(i \to j)R^\ast(i \to j) - 3L_n(L_n + 1)^2}
    {L_n^3 - L_n}.
\end{equation} 

To link $\rho(G_n)$ to \eqref{eq:spearmans_rho_theoretical}, let $h_n$ denote the empirical joint probability density function of the degrees on both sides of a random edge, i.e.
\[
	h_n(k, \ell) = \frac{1}{L_n} \sum_{i \to j \in G} \hspace{-3pt} \ind{D_i = k} \ind{D_j = \ell}.
\]
Then, if $h_n$ converges to some limiting distribution $h$, it follows from Theorem 3.2 in \cite{Hofstad2014} that
\[
	\rho(G_n) \plim \rho(X,Y) \quad \text{as } n \to \infty,
\]
where $X$ and $Y$ have joint distribution $h$. In other words, $\rho(G_n)$ is a consistent estimator of $\rho(X,Y)$. Moreover, in~\cite{Hoorn2014a} it is shown that $\rho(G_n)$ is asymptotically equivalent to
\begin{equation}\label{eq:spearmans_rho}
	\widetilde{\rho}(G_n) = \frac{3}{L_n} \sum_{i \to j} \mathcal{F}_n^\ast(D_i)
    \mathcal{F}_n^\ast(D_j) - 3.
\end{equation}
Since this expression is easier to analyze mathematically, we will use this measures in our statements. We show with numerical experiments in Section \ref{sec:spearman_experiments} that our results also hold for the original expression \ref{eq:spearmans_rho_full}.

\subsection{Main results}\label{ssec:main_results}

In order to state the first result we define, for any $k, \ell \ge 1$, the functions
\begin{align}
	\psi(k, \ell) &= \ind{1 - F^\ast(k) < F^\ast(\ell)} \ind{1 - F^\ast(k - 1) > F^\ast(\ell - 1)}, 
    \label{eq:def_psi}\\
    \mathcal{E}(k,\ell) &= \min\left(1 - F^\ast(k - 1), F^\ast(\ell)\right) -\max\left(1-F^\ast(k), 
    F^\ast(\ell-1)\right) \label{eq:def_mathcal_E}.
\end{align} 

These functions can be understood as follows. Consider the partition of the interval $[0,1]$ given by the sequence $\{f^\ast(1), f^\ast(2), \dots\}$. Now take a copy of this partitioned interval, reverse it and place it below the original interval, see Figure \ref{fig:explanation_joint_degree_distribution}. Then $\psi(k,\ell)$ is the indicator of the event that the interval corresponding to $f^\ast(\ell)$ on the top intersects with the interval corresponding to $f^\ast(k)$ at the bottom, while $\mathcal{E}(k,\ell)$ is the size of this intersection. 

\begin{figure}
	\centering
    \begin{tikzpicture}
		\draw (0,2) -- (10,2);
        \draw (0,1.8) -- (0,2.2);
        \draw node at (0.6,2.5) {$f^\ast(1)$};
        \draw (1.2,1.8) -- (1.2,2.2);
        \draw node at (1.9,2.5) {$f^\ast(2)$};
        \draw (2.6,1.8) -- (2.6,2.2);
        \draw node at (3.1,2.5) {$\dots$};
        \draw (3.6,1.8) -- (3.6,2.2);
        \draw node at (4.2,2.5) {$f^\ast(\ell)$};
        \draw (4.8,1.8) -- (4.8,2.2);
        \draw (10,1.8) -- (10,2.2);
        
        \draw [dotted] (3.6,0.3) -- (3.6,1.7);
        \draw [dotted] (4.2,0.3) -- (4.2,1.7);
        
        \draw (0,0) -- (10,0);
        \draw (0,-0.2) -- (0,0.2);
        \draw (3,-0.2) -- (3,0.2);
        \draw node at (3.6,-0.5) {$f^\ast(k)$};
        \draw (4.2,-0.2) -- (4.2,0.2);
        \draw node at (5.8,-0.5) {$\dots$};
        \draw (7.4,-0.2) -- (7.4,0.2);
        \draw node at (8.1,-0.5) {$f^\ast(2)$};
        \draw (8.8,-0.2) -- (8.8,0.2);
        \draw node at (9.4,-0.5) {$f^\ast(1)$};
        \draw (10,-0.2) -- (10,0.2);
	\end{tikzpicture}
    \caption{Illustration of the functions $\psi$ and $\mathcal{E}$.}
    \label{fig:explanation_joint_degree_distribution}
\end{figure}
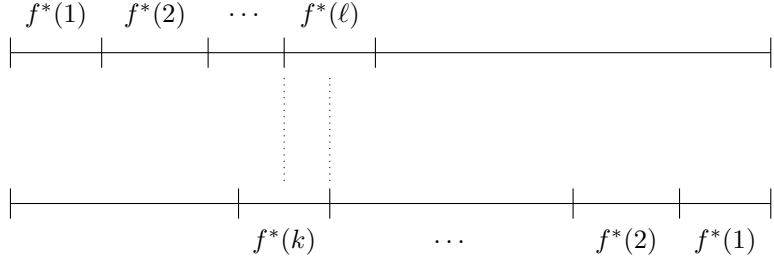

With these functions we now define the joint probability density function
\begin{equation}\label{eq:joint_degree_density}
	h(k, \ell) = \psi(k, \ell)\mathcal{E}(k,\ell), \quad k, \, \ell = 1, \, 2, \, \dots.
\end{equation}
Our main result states that if $X$ and $Y$ have joint distribution $h$, then Spearman's rho on graphs with a degree sequence satisfying Assumption \ref{asmp:regularity_distributions} is bounded from below by $\rho(X,Y)$, and that this minimum is attained for a specific sequence of graphs.

\begin{theorem}\label{thm:main_result}
Let $G_n \in \mathcal{G}_{\eta,\,\varepsilon}(f, f^\ast)$ and let $X, Y$ be random variables with joint distribution $h$ as defined in \eqref{eq:joint_degree_density}. 
Then, for any $0 < \delta < \min(\varepsilon, 1/2)$ and $K > 0$,
\begin{equation}\label{eq:ch6_convergence_lower_bound_spearman}
	\Prob{\widetilde{\rho}(G_n) \ge \widetilde{\rho}(D_\ast,D^\ast)- K n^{-\delta}} \ge 1 
	- \bigO{n^{-\varepsilon + \kappa} + \Prob{\Omega_n^c}},
\end{equation}
as $n \to \infty$, where
\[
	\kappa = \frac{\varepsilon + \delta}{2}.
\]

Moreover, there exists graphs $\widehat{G}_n$ with the same degree sequence as 
$G_n$, such that, as $n \to \infty$,
\[
	\Prob{\left|\widetilde{\rho}(\widehat{G}_n) - \rho(D_\ast,D^\ast)\right| > K n^{-\delta} } \le 
   \bigO{n^{-\varepsilon + \kappa} + \Prob{\Omega_n^c}}.
\]
\end{theorem} 

This result can be understood in terms of the following optimization problem. Given a degree sequences ${\bf D}_n \in \mathcal{D}_{\eta,\,\varepsilon}(f, f^\ast)$,
define
\[
	\mathcal{F}_n^\ast(k) = F_n^\ast(k) + F_n^\ast(k - 1).
\] 
and consider, for fixed $n$, the following objective function
\begin{equation}\label{eq:approx_optimization_problem}
	\min_{G \in \mathcal{G}({\bf D}_n)(f,f^\ast)} \frac{1}{L_n} \sum_{i \to j \in G} \mathcal{F}_n^\ast(D_i)\mathcal{F}_n^\ast(D_j),
\end{equation}
where the minimum is understood to be taken over all graphs $G_n$ with degree sequences satisfying Assumption \ref{asmp:regularity_distributions} with density functions
$f$ and $f^\ast$. Then Theorem \ref{thm:main_result} states that with high probability
\[
	\min_{G_n \in \mathcal{G}_{\eta,\,\varepsilon}(f, f^\ast)} \rho(G_n) = \rho(D_\ast,D^\ast),
\]
where $\rho(D_\ast, D^\ast)$ is given by, see \eqref{eq:spearmans_rho_theoretical},
\[
	\rho(D_\ast, D^\ast) = \sum_{k, \ell = 0}^\infty \mathcal{F}^\ast(k)\mathcal{F}^\ast(\ell) 
		\psi(k,\ell)\mathcal{E}(k,\ell),
\]
with $\psi$ and $\mathcal{E}$ as defined in \eqref{eq:def_psi} and \eqref{eq:def_mathcal_E},respectively. Moreover, Theorem \ref{thm:main_result}
provides a sequence of graphs $\widehat{G}_n$ that attains this minimum, i.e. a sequence of maximally disassortative graphs with the degree distribution $f$. These graphs will be generated by our algorithm, which we will present in Section \ref{sec:algorithm}.

We remark that although Theorem \ref{thm:main_result} solves the minimization problem of degree-degree 
correlations in undirected graphs by giving an the asymptotic minimum $\rho(D_\ast, D^\ast)$ on 
Spearman's rho, this minimum is, in general, hard to derive since it depends on the full 
size-biased limit density $f^\ast$. However, for specific cases it can be computed numerically by 
computing
\[
	\sum_{k = 0}^{K} \sum_{\ell = 0}^{L} \mathcal{F}^\ast(k)\mathcal{F}^\ast(\ell) \psi(k,\ell)\mathcal{E}(k,\ell),
\]
for certain upper bounds $K$ and $L$. 

Part of the proof of Theorem \ref{thm:main_result} consists of showing that $h$ is the limiting joint degree distribution of maximally disassortative graphs. From the interpretation of the functions $\psi$ and $\mathcal{E}$, it follows that for all $k \ge K$, for some threshold $K$, all intervals corresponding to $f^\ast(k)$ on the top will be contained in the interval $f^\ast(1)$ at the bottom and vice versa. This implies that the large degree nodes will, asymptotically, all be connected to nodes with degree one. As a consequence we have that the tail of the distribution $f^\ast$, and hence also that of $f$, plays a negligible role in the lower bound of Spearman's rho. Therefore, we can construct degree distributions with specified tail behavior so that Spearman's rho on such graphs approaches zero from below, with arbitrary precision. 

\begin{theorem}\label{thm:independence_spearman_tail}
Let $f$ be any probability density function with support on the non-negative integers, mean 
$\mu$ and 
\[
	\sum_{k = 0}^\infty k^{1 + \eta} f(k) < \infty,
\]
for some $\eta > 0$. Then, for any $-1 < \rho < 0$, there exists a probability density function 
$f_\rho$ on the non-negative integers with mean $\mu_\rho$, which satisfies 
\[
	\lim_{k \to \infty} \frac{1 - F_\rho(k)}{1 - F(k)} = \frac{\mu_\rho}{\mu}.
\]
Moreover, for any sequence of graphs $G_n \in \mathcal{G}_{\eta, \, \varepsilon}(f_\rho, f_\rho^\ast)$, where
$f_\rho^\ast(k) = kf_\rho(k)/\mu_\rho$, we have
\[
	\Prob{\widetilde{\rho}(G_n) > \rho} \ge 1 - \bigO{n^{-1 + \kappa} + n^{-\varepsilon + 3\kappa/4}
	+ \Prob{\Omega_n^c}},
\]
as $n \to \infty$, where $\kappa = \min(\varepsilon, 1/2)$.
\end{theorem}
The main message of Theorem \ref{thm:independence_spearman_tail} is that it is not the tail of 
the degree distribution that is crucial for the minimal value of $\widetilde{\rho}(G_n)$.

The characterization of the tail of the degree distribution is most prominently present in the analysis of so-called scale-free networks. These are graphs where the limiting degree distribution $F$ satisfies
\begin{equation}\label{eq:regularly_varying_degrees}
	1 - F(k) = \mathcal{L}(k) k^{-\gamma}, \quad \gamma > 1,
\end{equation}
for some slowly varying function $\mathcal{L}$. The exponent $\gamma$ is referred to as the tail exponent. As a corollary to Theorem \ref{thm:independence_spearman_tail} we obtain the following result which states that knowledge of only the tail exponent does not give any guarantees on the minimum value of Spearman's rho. 

\begin{corollary}\label{cor:independence_rho_scale_free}
For any $-1 < \rho < 0$ and $\gamma > 1$, there exist distributions $f$ and $f^\ast$, where $F$ satisfies \eqref{eq:regularly_varying_degrees}, such that for any sequence of graphs $G_n \in 
\mathcal{G}_{\eta, \varepsilon}(f, f^\ast)$
\[
	\Prob{\rho(G_n) > \rho} \ge 1 - \bigO{n^{-1 + \kappa} + n^{-\varepsilon + 3\kappa/4}
	+ \Prob{\Omega_n^c}},
\]
as $n \to \infty$, where $\kappa = \min(\varepsilon, 1/2)$.
\end{corollary}

\subsection{Structure of the paper}

The rest of this paper is structured as follows. In Section \ref{ssec:algorithm} we describe the algorithm for generating graphs that solves the optimization problem \eqref{eq:approx_optimization_problem}. A complete characterization of the empirical and limiting joint degree distribution is then given in Section \ref{ssec:joint_degree_distribution}. We describe the construction of degree sequences with arbitrary small value of Spearman's rho in Section \ref{sec:spearman_tail_distribution}. In Section \ref{sec:spearman_experiments} we illustrate our results by providing simulations for maximally disassortative graphs where the degrees follow a scale-free and a Poisson distribution. Finally, Section \ref{sec:proofs} contains all the proofs of our results.

\section{Generating maximally disassortative graphs}\label{sec:algorithm}

We will describe an algorithm, called the \texttt{Disassortative Graph Algorithm} (\texttt{DGA}), that solves~\eqref{eq:approx_optimization_problem}. 

\subsection{The Disassortative Graph Algorithm}\label{ssec:algorithm}

Any degree sequence ${\bf D}_n$ can be represented by a list of stubs, where for each node $i$ we have $D_i$ stubs labeled $i$. A graph with degree sequence ${\bf D}_n$ is then completely determined by the pairing of the stubs. In order to describe our algorithm, let $N_k$ denote the number of nodes with degree $k$ and let $z_n$ be the unique integer satisfying
\begin{equation}\label{eq:relation_zn}
	\sum_{t = 1}^{z_n} tN_t \ge \frac{L_n}{2} \quad \text{and} \quad \sum_{t = 1}^{z_n - 1} tN_t < \frac{L_n}{2}.
\end{equation}   
The idea of the \texttt{Disassortative Graph Algorithm} is to use $z_n$ to divide the stubs in two 
columns. In the left column $S_n$ we add the stubs belonging to nodes with high degree ($D_i \ge z_n$), 
in descending order. The right column $T_n$ will be filled with stubs that belong to nodes with small 
degree ($D_i \le z_n$) in ascending order. After this ordering we start pairing stubs from the left column to stubs in the right column, until we reach the first pair $(i,j)$ for which $D_i = z_n = D_j$. We are now left with stubs belonging to nodes with degree $z_n$, hence the value of Spearman's rho \eqref{eq:spearmans_rho} will not be influenced by the specific way in which we connect them. This means that we can, in principle, use any algorithm to connect these medium degree nodes. We will use the  configuration model \cite{Bollobas1980,Molloy1995,Newman2001}, more specifically the repeated configuration model, see Section 7.4 in \cite{Hofstad2014a}. The full algorithm is described in detail below.

\begin{algorithm*}[!ht]
	\caption{\texttt{Disassortative Graph Algorithm}}
	\label{alg:ch6_disassortative_graphs}
	\begin{algorithmic}[1]
		\STATE Input: A degree seqeunce ${\bf D}_n$.
		\STATE Rank the nodes by their degrees in ascending order and let $\varrho(k)$ and denote the 
			node with rank $k$, i.e. $D_{\phi(n)} \ge D_{\phi(n-1)} \ge \dots \ge D_{\phi(2)} 
			\ge D_{\phi(1)}$.
		\STATE Create two empty lists $S_n$ and $T_n$.
		\STATE Set $i = n$ and $j = 1$.
		\WHILE{$D_{\phi(i)} \ge z_n$}
			\STATE Fill the next $D_{\phi(i)}$ slots of $S_n$ with stubs labeled $\phi(i)$.
			\STATE Set $i = i - 1$.
		\ENDWHILE
		\WHILE{$D_{\phi(j)} \le z_n$}
			\STATE Add to $T_n$, $D_{\phi(j)}$ copies of stubs labeled: $\phi(j), \dots, 
				\varrho(j+N_{D_{\phi(j)}}-1)$.
			\STATE Set $j = j+ N_{D_{\phi(j)}}$.
		\ENDWHILE
		\STATE Set $t = 1$, $i = S_n[t]$ and $j = T_n[t]$
		\WHILE{\NOT $D_i = z_n = D_j$} 
			\STATE Add edge $i - j$ to $G_n$.
			\STATE Set $t = t + 1$, $i = S_n[t]$ and $j = T_n[t]$.
		\ENDWHILE
		\STATE Set ${\bf D}^z_n$ to be the degree sequence corresponding to the remaining unpaired 
			stubs.
		\STATE Pair the stubs in ${\bf D}^z_n$ using the configuration model.
		\STATE Output: $G_n$.
	\end{algorithmic}
\end{algorithm*} 

We will denote by $G_n^\ast$ the induced sub-graph that has been created at the end of step 17 and 
let the compliment $G_n^z = G_n \setminus G_n^\ast$ denote the graph generated by the configuration
model in step 19. In addition we will write $G_n = \texttt{DGA}({\bf D}_n)$ if
$G_n$ is generated by the \texttt{Disassortative Graph Algorithm} with degree sequence ${\bf D}_n$
as input.  An illustration of the lists $S_n$ and $T_n$ is displayed in Figure 
\ref{fig:top_part_ordered_stubs}. 

\begin{figure}[!ht]
\centering
\begin{tikzpicture}

	\draw (0,2) -- (0,10) -- (2,10) -- (2,2);
	
    \draw node at (1,9.75) {$\phi(n)$};
    \draw (0,9.5) -- (2,9.5);
    
    \draw (0,8.5) -- (2,8.5);
    \draw node at (1,8.25) {$\phi(n)$};
	\draw [line width=2pt] (0,8) -- (2,8);
    \draw [decorate,decoration={brace},line width=1pt] (-0.2,8) -- (-0.2,10);
    \draw node at (-1,9) {$D_{\phi(n)}$};
    
    \draw node at (1,7.75) {$\phi(n-1)$};
    \draw (0,7.5) -- (2,7.5);
    
    \draw (0,6.5) -- (2,6.5);
    \draw node at (1,6.25) {$\phi(n-1)$};
    \draw [line width=2pt] (0,6) -- (2,6);
    \draw [decorate,decoration={brace},line width=1pt] (-0.2,6) -- (-0.2,8);
    \draw node at (-1,7) {$D_{\phi(n - 1)}$};
    
    \draw node at (1,5.75) {$\phi(n-2)$};
    \draw (0,5.5) -- (2,5.5);
    
    \draw (0,5) -- (2,5);
    \draw node at (1,4.75) {$\phi(n-2)$};
    \draw [line width=2pt] (0,4.5) -- (2,4.5);
    \draw [decorate,decoration={brace},line width=1pt] (-0.2,4.5) -- (-0.2,6);
    \draw node at (-1,5.25) {$D_{\phi(n - 2)}$};
    
	\draw (4,2) -- (4,10) -- (7,10) -- (7,2);
	
    \draw node at (5.5,9.75) {$\phi(1)$};
    \draw (4,9.5) -- (7,9.5);
    
    \draw (4,7.5) -- (7,7.5);
    \draw node at (5.5,7.25) {$\phi(N_1)$};
    \draw [line width=2pt] (4,7) -- (7,7);
    
    \draw [decorate,decoration={brace},line width=1pt] (7.2,10) -- (7.2,7);
    \draw node at (7.75,8.5) {$N_1$};
    
    \draw node at (5.5,6.75) {$\phi(N_1 + 1)$};
    \draw (4,6.5) -- (7,6.5);
    \draw (4,5.5) -- (7,5.5);
    \draw node at (5.5,5.25) {$\phi(N_1 + N_2 - 1)$};
    \draw [line width=2pt] (4,5) -- (7,5);
    
    \draw [decorate,decoration={brace},line width=1pt] (7.2,7) -- (7.2,5);
    \draw node at (7.75,6) {$N_2$};
    
    \draw node at (5.5,4.75) {$\phi(N_1 + 1)$};
    \draw (4,4.5) -- (7,4.5);
    \draw (4,3.5) -- (7,3.5);
    \draw node at (5.5,3.25) {$\phi(N_1 + N_2 - 1)$};
    \draw [line width=2pt] (4,3) -- (7,3);
    
    \draw [decorate,decoration={brace},line width=1pt] (7.2,5) -- (7.2,3);
    \draw node at (7.75,4) {$N_2$};
    
\end{tikzpicture}
\caption{Top part of the two lists of stubs.}
\label{fig:top_part_ordered_stubs}

\end{figure}
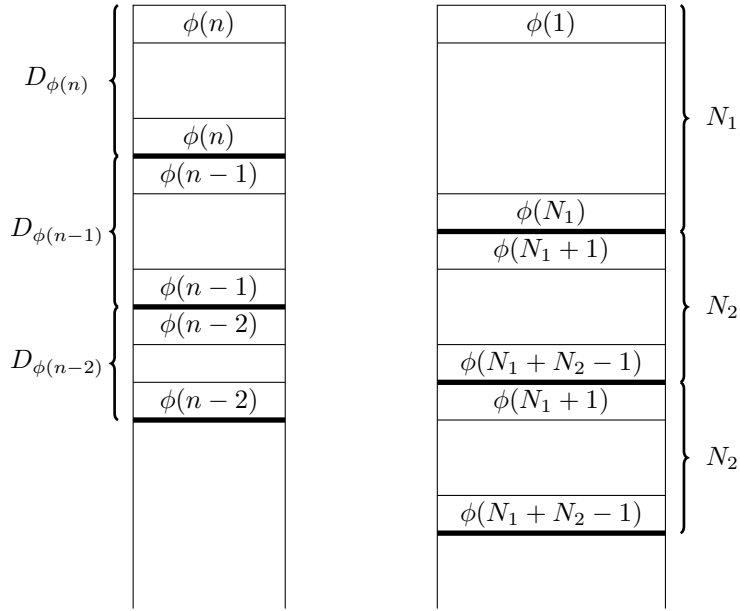

We will illustrate the \texttt{DGA} on the simple degree sequence $\{1,2,2,3\}$, see Figure 
\ref{fig:example_dga}. Observe that in this case $z_n = 2$. Figure \ref{sfig:dga_example_1} shows the initialization state where
the we have created the the lists $S_n$ and $T_n$ and no stubs have been connected. We start,
Figure \ref{sfig:dga_example_2}, by connecting the nodes at the top of the lists, $4$ and $1$. Then
we move down the lists, Figure \ref{sfig:dga_example_3}, and connect $4$ and $2$. The next step, 
Figure \ref{sfig:dga_example_4}, is where the specific way the algorithm ordered the stubs in both 
lists comes into play. 

There is one stub left on the node with the largest degree, node $4$. The 
smallest degree among the still available nodes is two. Therefore we want to connect node $4$ to
a node with degree two which are $2$ and $3$. However, since there is already an edge between $4$ 
and $2$, connecting them again will create multi-edges between these nodes. The ordering of the 
lists resolves this by making sure we first connected to each different node with the same degree 
before we can create an edge between two nodes that have already been connected. In this example
we therefore connect $4$ and $3$.

After this step the algorithm reaches a pair of nodes that both have degree $z_n = 2$, Figure 
\ref{sfig:dga_example_5}. This is where we stop and pair the remaining stubs using the 
configuration model. Since in this specific example only nodes $2$ and $3$ have a stub left,
we connect these, Figure \ref{sfig:dga_example_6}.

\begin{figure}[!ht]
	\centering
	\begin{subfigure}{0.48\textwidth}
		\centering
		\tikzstyle{vertex}=[fill, circle, minimum size=2pt]
\tikzstyle{edge}=[style=thick,line width=1,color=black]

\begin{tikzpicture}[scale=0.5]
	
	\draw node[vertex] (v1) at (4,3) {};
	\draw node[vertex] (v2) at (4,0) {};
	\draw node[vertex] (v3) at (0,0) {};
	\draw node[vertex] (v4) at (0,3) {};
	
	\draw node at (5,3) {\large $1$};
	\draw node at (5,0) {\large $2$};
	\draw node at (-1,0) {\large $3$};
	\draw node at (-1,3) {\large $4$};

	\draw [-][edge] (3,3) -- (v1);
	\draw [-][edge] (3,0.707) -- (v2);
	\draw [-][edge] (3,-0.707) -- (v2);
	\draw [-][edge] (v3) -- (1,0.707);
	\draw [-][edge] (v3) -- (1,-0.707);
	\draw [-][edge] (v4) -- (1,3.707);
	\draw [-][edge] (v4) -- (1,3);
	\draw [-][edge] (v4) -- (1,2.292);
	
	\draw (6,5) -- (7,5) -- (7,-2) -- (6,-2) -- (6,5);
	\draw node at (6.5,4.5) {\large $4$};
	\draw (6,4) -- (7,4);
	\draw node at (6.5,3.5) {\large $4$};
	\draw (6,3) -- (7,3);
	\draw node at (6.5,2.5) {\large $4$};
	\draw (6,2) -- (7,2);
	\draw node at (6.5,1.5) {\large $3$};
	\draw (6,1) -- (7,1);
	\draw node at (6.5,0.5) {\large $3$};
	\draw (6,0) -- (7,0);
	\draw node at (6.5,-0.5) {\large $2$};
	\draw (6,-1) -- (7,-1);
	\draw node at (6.5,-1.5) {\large $2$};
	
	\draw (8,5) -- (9,5) -- (9,0) -- (8,0) -- (8,5);
	\draw node at (8.5,4.5) {\large $1$};
	\draw (8,4) -- (9,4);
	\draw node at (8.5,3.5) {\large $2$};
	\draw (8,3) -- (9,3);
	\draw node at (8.5,2.5) {\large $3$};
	\draw (8,2) -- (9,2);
	\draw node at (8.5,1.5) {\large $2$};
	\draw (8,1) -- (9,1);
	\draw node at (8.5,0.5) {\large $3$};
	
\end{tikzpicture}
		\caption{Initialization}
		\label{sfig:dga_example_1}
	\end{subfigure}
	\begin{subfigure}{0.48\textwidth}
		\centering
		\tikzstyle{vertex}=[fill, circle, minimum size=2pt]
\tikzstyle{edge}=[style=thick,line width=1,color=black]

\begin{tikzpicture}[scale=0.5]
	
	\draw node[vertex] (v1) at (4,3) {};
	\draw node[vertex] (v2) at (4,0) {};
	\draw node[vertex] (v3) at (0,0) {};
	\draw node[vertex] (v4) at (0,3) {};
	
	\draw node at (5,3) {\large $1$};
	\draw node at (5,0) {\large $2$};
	\draw node at (-1,0) {\large $3$};
	\draw node at (-1,3) {\large $4$};

	\draw [-][edge] (v4) -- (v1);
	\draw [-][edge] (3,0.707) -- (v2);
	\draw [-][edge] (3,-0.707) -- (v2);
	\draw [-][edge] (v3) -- (1,0.707);
	\draw [-][edge] (v3) -- (1,-0.707);
	\draw [-][edge] (v4) -- (1,3.707);
	\draw [-][edge] (v4) -- (1,2.292);
	
	\draw (6,5) -- (7,5) -- (7,-2) -- (6,-2) -- (6,5);
	\draw node at (6.5,4.5) {\large $4$};
	\draw (6,4) -- (7,4);
	\draw node at (6.5,3.5) {\large $4$};
	\draw (6,3) -- (7,3);
	\draw node at (6.5,2.5) {\large $4$};
	\draw (6,2) -- (7,2);
	\draw node at (6.5,1.5) {\large $3$};
	\draw (6,1) -- (7,1);
	\draw node at (6.5,0.5) {\large $3$};
	\draw (6,0) -- (7,0);
	\draw node at (6.5,-0.5) {\large $2$};
	\draw (6,-1) -- (7,-1);
	\draw node at (6.5,-1.5) {\large $2$};
	
	\draw [<->][edge] (7,4.5) -- (8,4.5);
	
	\draw (8,5) -- (9,5) -- (9,0) -- (8,0) -- (8,5);
	\draw node at (8.5,4.5) {\large $1$};
	\draw (8,4) -- (9,4);
	\draw node at (8.5,3.5) {\large $2$};
	\draw (8,3) -- (9,3);
	\draw node at (8.5,2.5) {\large $3$};
	\draw (8,2) -- (9,2);
	\draw node at (8.5,1.5) {\large $2$};
	\draw (8,1) -- (9,1);
	\draw node at (8.5,0.5) {\large $3$};
	
\end{tikzpicture}
		\caption{First step}
		\label{sfig:dga_example_2}
	\end{subfigure}\\
	\vspace{5pt}
	\begin{subfigure}{0.48\textwidth}
		\centering
		\tikzstyle{vertex}=[fill, circle, minimum size=2pt]
\tikzstyle{edge}=[style=thick,line width=1,color=black]

\begin{tikzpicture}[scale=0.5]
	
	\draw node[vertex] (v1) at (4,3) {};
	\draw node[vertex] (v2) at (4,0) {};
	\draw node[vertex] (v3) at (0,0) {};
	\draw node[vertex] (v4) at (0,3) {};
	
	\draw node at (5,3) {\large $1$};
	\draw node at (5,0) {\large $2$};
	\draw node at (-1,0) {\large $3$};
	\draw node at (-1,3) {\large $4$};

	\draw [-][edge] (v4) -- (v1);
	\draw [-][edge] (v4) -- (v2);
	\draw [-][edge] (3,-0.707) -- (v2);
	\draw [-][edge] (v3) -- (1,0.707);
	\draw [-][edge] (v3) -- (1,-0.707);
	\draw [-][edge] (v4) -- (1,3.707);
	
	\draw (6,5) -- (7,5) -- (7,-2) -- (6,-2) -- (6,5);
	\draw node at (6.5,4.5) {\large $4$};
	\draw (6,4) -- (7,4);
	\draw node at (6.5,3.5) {\large $4$};
	\draw (6,3) -- (7,3);
	\draw node at (6.5,2.5) {\large $4$};
	\draw (6,2) -- (7,2);
	\draw node at (6.5,1.5) {\large $3$};
	\draw (6,1) -- (7,1);
	\draw node at (6.5,0.5) {\large $3$};
	\draw (6,0) -- (7,0);
	\draw node at (6.5,-0.5) {\large $2$};
	\draw (6,-1) -- (7,-1);
	\draw node at (6.5,-1.5) {\large $2$};
	
	\draw [<->][edge] (7,3.5) -- (8,3.5);
	
	\draw (8,5) -- (9,5) -- (9,0) -- (8,0) -- (8,5);
	\draw node at (8.5,4.5) {\large $1$};
	\draw (8,4) -- (9,4);
	\draw node at (8.5,3.5) {\large $2$};
	\draw (8,3) -- (9,3);
	\draw node at (8.5,2.5) {\large $3$};
	\draw (8,2) -- (9,2);
	\draw node at (8.5,1.5) {\large $2$};
	\draw (8,1) -- (9,1);
	\draw node at (8.5,0.5) {\large $3$};
	
\end{tikzpicture}
		\caption{Second step}
		\label{sfig:dga_example_3}
	\end{subfigure}
	\begin{subfigure}{0.48\textwidth}
		\centering
		\tikzstyle{vertex}=[fill, circle, minimum size=2pt]
\tikzstyle{edge}=[style=thick,line width=1,color=black]

\begin{tikzpicture}[scale=0.5]
	
	\draw node[vertex] (v1) at (4,3) {};
	\draw node[vertex] (v2) at (4,0) {};
	\draw node[vertex] (v3) at (0,0) {};
	\draw node[vertex] (v4) at (0,3) {};
	
	\draw node at (5,3) {\large $1$};
	\draw node at (5,0) {\large $2$};
	\draw node at (-1,0) {\large $3$};
	\draw node at (-1,3) {\large $4$};

	\draw [-][edge] (v4) -- (v1);
	\draw [-][edge] (v4) -- (v2);
	\draw [-][edge] (3,-0.707) -- (v2);
	\draw [-][edge] (v3) -- (1,-0.707);
	\draw [-][edge] (v4) -- (v3);
	
	\draw (6,5) -- (7,5) -- (7,-2) -- (6,-2) -- (6,5);
	\draw node at (6.5,4.5) {\large $4$};
	\draw (6,4) -- (7,4);
	\draw node at (6.5,3.5) {\large $4$};
	\draw (6,3) -- (7,3);
	\draw node at (6.5,2.5) {\large $4$};
	\draw (6,2) -- (7,2);
	\draw node at (6.5,1.5) {\large $3$};
	\draw (6,1) -- (7,1);
	\draw node at (6.5,0.5) {\large $3$};
	\draw (6,0) -- (7,0);
	\draw node at (6.5,-0.5) {\large $2$};
	\draw (6,-1) -- (7,-1);
	\draw node at (6.5,-1.5) {\large $2$};
	
	\draw [<->][edge] (7,2.5) -- (8,2.5);
	
	\draw (8,5) -- (9,5) -- (9,0) -- (8,0) -- (8,5);
	\draw node at (8.5,4.5) {\large $1$};
	\draw (8,4) -- (9,4);
	\draw node at (8.5,3.5) {\large $2$};
	\draw (8,3) -- (9,3);
	\draw node at (8.5,2.5) {\large $3$};
	\draw (8,2) -- (9,2);
	\draw node at (8.5,1.5) {\large $2$};
	\draw (8,1) -- (9,1);
	\draw node at (8.5,0.5) {\large $3$};
	
\end{tikzpicture}
		\caption{Third step}
		\label{sfig:dga_example_4}
	\end{subfigure}\\
	\vspace{5pt}
	\begin{subfigure}{0.48\textwidth}
		\centering
		\tikzstyle{vertex}=[fill, circle, minimum size=2pt]
\tikzstyle{edge}=[style=thick,line width=1,color=black]

\begin{tikzpicture}[scale=0.5]
	
	\draw node[vertex] (v1) at (4,3) {};
	\draw node[vertex] (v2) at (4,0) {};
	\draw node[vertex] (v3) at (0,0) {};
	\draw node[vertex] (v4) at (0,3) {};
	
	\draw node at (5,3) {\large $1$};
	\draw node at (5,0) {\large $2$};
	\draw node at (-1,0) {\large $3$};
	\draw node at (-1,3) {\large $4$};

	\draw [-][edge] (v4) -- (v1);
	\draw [-][edge] (v4) -- (v2);
	\draw [-][edge] (3,-0.707) -- (v2);
	\draw [-][edge] (v3) -- (1,-0.707);
	\draw [-][edge] (v4) -- (v3);
	
	\draw (6,5) -- (7,5) -- (7,-2) -- (6,-2) -- (6,5);
	\draw node at (6.5,4.5) {\large $4$};
	\draw (6,4) -- (7,4);
	\draw node at (6.5,3.5) {\large $4$};
	\draw (6,3) -- (7,3);
	\draw node at (6.5,2.5) {\large $4$};
	\draw (6,2) -- (7,2);
	\draw node at (6.5,1.5) {\large $3$};
	\draw (6,1) -- (7,1);
	\draw node at (6.5,0.5) {\large $3$};
	\draw (6,0) -- (7,0);
	\draw node at (6.5,-0.5) {\large $2$};
	\draw (6,-1) -- (7,-1);
	\draw node at (6.5,-1.5) {\large $2$};
	
	\draw [<->][edge] (7,1.5) -- (8,1.5);
	
	\draw (8,5) -- (9,5) -- (9,0) -- (8,0) -- (8,5);
	\draw node at (8.5,4.5) {\large $1$};
	\draw (8,4) -- (9,4);
	\draw node at (8.5,3.5) {\large $2$};
	\draw (8,3) -- (9,3);
	\draw node at (8.5,2.5) {\large $3$};
	\draw (8,2) -- (9,2);
	\draw node at (8.5,1.5) {\large $2$};
	\draw (8,1) -- (9,1);
	\draw node at (8.5,0.5) {\large $3$};
	
\end{tikzpicture}
		\caption{Reached $z_n$}
		\label{sfig:dga_example_5}
	\end{subfigure}
	\begin{subfigure}{0.48\textwidth}
		\centering
		\tikzstyle{vertex}=[fill, circle, minimum size=2pt]
\tikzstyle{edge}=[style=thick,line width=1,color=black]

\begin{tikzpicture}[scale=0.5]
	
	\draw node[vertex] (v1) at (4,3) {};
	\draw node[vertex] (v2) at (4,0) {};
	\draw node[vertex] (v3) at (0,0) {};
	\draw node[vertex] (v4) at (0,3) {};
	
	\draw node at (5,3) {\large $1$};
	\draw node at (5,0) {\large $2$};
	\draw node at (-1,0) {\large $3$};
	\draw node at (-1,3) {\large $4$};

	\draw [-][edge] (v4) -- (v1);
	\draw [-][edge] (v4) -- (v2);
	\draw [-][edge] (v3) -- (v2);
	\draw [-][edge] (v4) -- (v3);
	
	\draw (6,5) -- (7,5) -- (7,-2) -- (6,-2) -- (6,5);
	\draw node at (6.5,4.5) {\large $4$};
	\draw (6,4) -- (7,4);
	\draw node at (6.5,3.5) {\large $4$};
	\draw (6,3) -- (7,3);
	\draw node at (6.5,2.5) {\large $4$};
	\draw (6,2) -- (7,2);
	\draw node at (6.5,1.5) {\large $3$};
	\draw (6,1) -- (7,1);
	\draw node at (6.5,0.5) {\large $3$};
	\draw (6,0) -- (7,0);
	\draw node at (6.5,-0.5) {\large $2$};
	\draw (6,-1) -- (7,-1);
	\draw node at (6.5,-1.5) {\large $2$};
	
	\draw (8,5) -- (9,5) -- (9,0) -- (8,0) -- (8,5);
	\draw node at (8.5,4.5) {\large $1$};
	\draw (8,4) -- (9,4);
	\draw node at (8.5,3.5) {\large $2$};
	\draw (8,3) -- (9,3);
	\draw node at (8.5,2.5) {\large $3$};
	\draw (8,2) -- (9,2);
	\draw node at (8.5,1.5) {\large $2$};
	\draw (8,1) -- (9,1);
	\draw node at (8.5,0.5) {\large $3$};
	
\end{tikzpicture}
		\caption{Connected at random}
		\label{sfig:dga_example_6}
	\end{subfigure}
	\caption{Example of the \texttt{DGA} on a simple degree sequence with $z_n = 2$.}
	\label{fig:example_dga}
\end{figure}
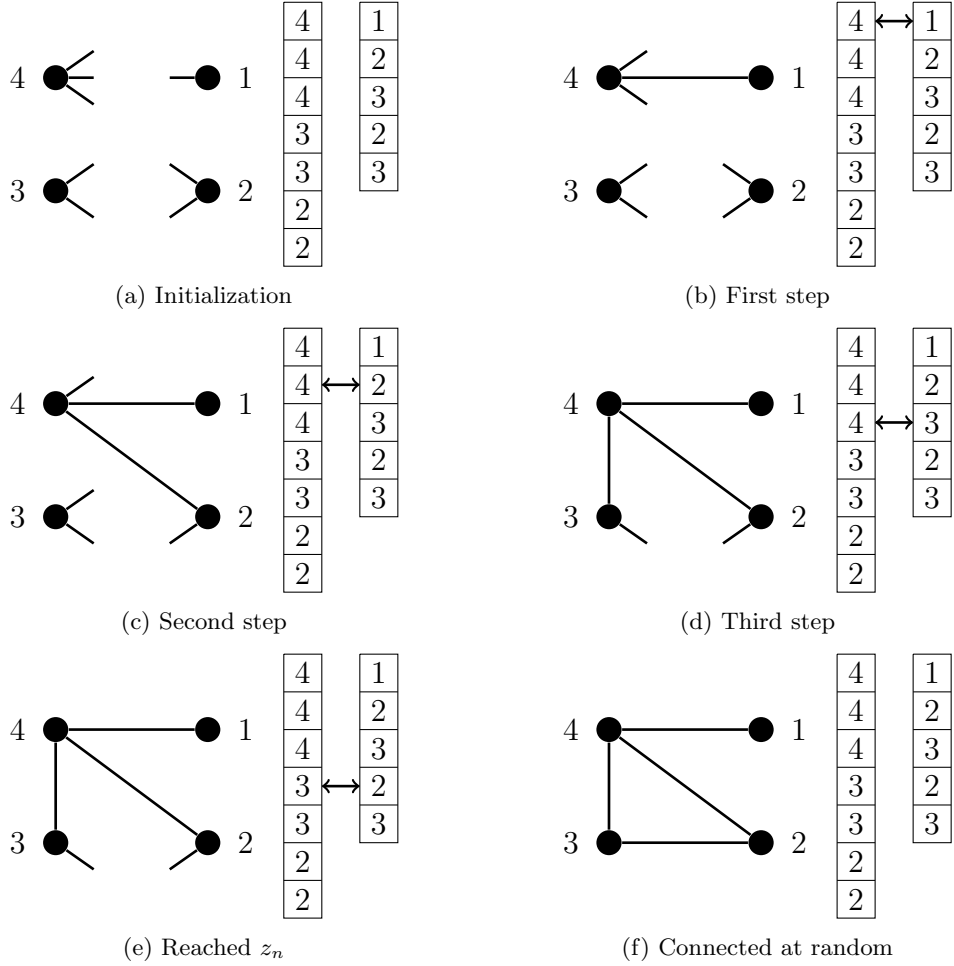

Although the \texttt{DGA} is defined for arbitrary degree sequences, in practice we would like to have ${\bf D}_n \in \mathcal{D}_{\eta, \,\varepsilon}(f, f^\ast)$ for some $\eta,\varepsilon > 0$ and distributions $f$ and $f^\ast$. A well known algorithm for generating degree sequences with a given distribution $f$ is by sampling the degrees i.i.d. from the distribution and then increase the last degree by $1$ if the sum was not even. We will refer to this as the \texttt{IID} algorithm. The following lemma states that when the distribution from which the degrees are sampled has just a bit more than finite mean, the resulting degree sequence satisfies Assumption \ref{asmp:regularity_distributions}.

\begin{lemma}\label{lem:iid_algorithm}
Let $D$ be an integer valued random variable with a distribution $f$, such that $\Exp{D^{1 + \eta}} < \infty$ for some $\eta > 0$. Denote by $\mu$ the mean of $f$ and define $f^\ast(k) = \Exp{D\ind{D = k}}/\mu$. Then if  ${\bf D}_n$ is generated by the \texttt{IID} algorithm by sampling from $f$, 
\[
	{\bf D}_n \in \mathcal{D}_{\eta, \, \varepsilon}(f, f^\ast) \quad \text{for any}
    \quad \varepsilon \le \eta/(8 + 4\eta).
\]
Moreover,
\[
	\Prob{\Omega_n} \ge 1 - \bigO{n^{-\varepsilon}},
\]
as $n \to \infty$.
\end{lemma}

\subsection{Joint degree distribution of maximally disassortative graphs} \label{ssec:joint_degree_distribution}

Before we turn to analysis of the \texttt{DGA} it is useful to look at the empirical joint degree distribution of graphs generated by the algorithm. We will give a complete characterization of both the empirical and limiting joint degree distributions in Proposition \ref{prop:empirical_joint_degree_distribution} and Theorem \ref{thm:joint_degree_distribution}, respectively.

In order to analyze the structure of the joint degree distribution we approach the algorithm from
a different angle. First observe that if we are only interested in the degrees $D_i$ and $D_j$ for an 
undirected edge $i - j$, then the specific way in which the stubs are ordered by the 
algorithm is irrelevant for $h^\ast_n(k,\ell)$. This means we do not have to consider the label of the nodes to which 
stubs belong, only their degree. Note that the number of stubs belonging to nodes of degree $k$ 
equals $k N_k$. 
Moreover, due to the symmetry in the transition to directed edges, by replacing an edge $i - j$ with 
edges $i \to j$ and $j \to i$, the directed structure of the graph generated by \texttt{DGA} can be seen 
as follows. 

Consider the partition of the set $\{1, 2, \dots, L_n\}$ given by $kN_k$ for $k = 0,1, 
\dots$, represented as a line of length $L_n$ partitioned into intervals of size $kN_k$. Now take 
a copy of this partitioned line, reverse it and place it below the original one, see Figure 
\ref{fig:edge_degree_structure_dga}. Both lines can be seen as the lists of all stubs, ordered
by the degree of the nodes to which they belong. For the top line the stubs are ordered, from left 
to right, in increasing order of the degree, while for the bottom line the degrees are in decreasing
order. Then the \texttt{DGA} can be seen as creating directed edges $i \to j$ between the nodes $i$ 
corresponding to the stubs on the bottom line and nodes $j$ corresponding to stubs in the top line.

\begin{figure}
	\centering
	\begin{tikzpicture}
		\draw (0,2) -- (14,2);
        \draw (0,1.8) -- (0,2.2);
        \draw node at (0.6,2.5) {$N_1$};
        \draw (1.2,1.8) -- (1.2,2.2);
        \draw node at (1.9,2.5) {$2 N_2$};
        \draw (2.6,1.8) -- (2.6,2.2);
        \draw node at (3.1,2.5) {$\dots$};
        \draw (3.6,1.8) -- (3.6,2.2);
        \draw node at (4.2,2.5) {$\ell N_\ell$};
        \draw (4.8,1.8) -- (4.8,2.2);
        \draw (5.8,1.8) -- (5.8,2.2);
        \draw node at (6.8,2.5) {$z_n N_{z_n}$};
        \draw (7.8,1.8) -- (7.8,2.2);
        \draw (9.8,1.8) -- (9.8,2.2);
        \draw node at (10.4,2.5) {$kN_k$};
        \draw (11,1.8) -- (11,2.2);
        \draw (14,1.8) -- (14,2.2);
        
        \draw [dotted] (3.6,0.3) -- (3.6,1.7);
        \draw node at (3.9,1.6) {$j$};
        \draw [->] (3.9,0.8) -- (3.9,1.2);
        \draw node at (3.9,0.4) {$i$};
        \draw [dotted] (4.2,0.3) -- (4.2,1.7);
        
        \draw (0,0) -- (14,0);
        \draw (0,-0.2) -- (0,0.2);
        \draw (3,-0.2) -- (3,0.2);
        \draw node at (3.6,-0.5) {$kN_k$};
        \draw (4.2,-0.2) -- (4.2,0.2);
         \draw node at (5.2,-0.5) {$\dots$};
        \draw (6.2,-0.2) -- (6.2,0.2);
        \draw node at (7.2,-0.5) {$z_n N_{z_n}$};
        \draw (8.2,-0.2) -- (8.2,0.2);

        \draw (9.2,-0.2) -- (9.2,0.2);
        \draw node at (9.8,-0.5) {$\ell N_\ell$};
        \draw (10.4,-0.2) -- (10.4,0.2);
        
        \draw (11.4,-0.2) -- (11.4,0.2);
        \draw node at (12.1,-0.5) {$2 N_2$};
        \draw (12.8,-0.2) -- (12.8,0.2);
        \draw node at (13.4,-0.5) {$N_1$};
        \draw (14,-0.2) -- (14,0.2);
	\end{tikzpicture}
    \caption{}
    \label{fig:edge_degree_structure_dga}
\end{figure}

From this representation we observe that an edge $i \to j$ between 
nodes of degree $D_i = k$ and $D_j = \ell$ exists if and only if the interval corresponding to 
$kN_k$ in the partitioned bottom line has an intersection with the interval corresponding to $\ell 
N_\ell$ in the partitioned upper line. In terms of $N_t$ this holds, if and only if,
\begin{equation}\label{eq:existence_edge_dga}
	\sum_{t = k + 1}^\infty tN_t < \sum_{t = 1}^\ell tN_t \quad \text{and} \quad
  \sum_{t = 1}^{\ell - 1} tN_t < \sum_{t = k}^\infty tN_t.
\end{equation}
Moreover, the number of edges that connect nodes of degree $k$ and $\ell$ is equal to the 
size of the intersection,
\begin{equation}\label{eq:number_edges_dga}
	\min\left\{\sum_{t = k}^\infty tN_t, \sum_{t = 1}^\ell tN_t\right\} 
	- \max\left\{\sum_{t = k + 1}^\infty tN_t, \sum_{t = 1}^{\ell - 1} tN_t\right\}.
\end{equation}

This partitioned representation of both the \texttt{DGA} and the joint degree structure, as displayed 
in Figure \ref{fig:edge_degree_structure_dga}, will be crucial for the analysis of the 
structure of maximally disassortative graphs. 

First let $\psi_n(k,\ell)$ denote the indicator that there exists a directed edge $i \to j$ with 
$D_i = k$ and $D_j = \ell$. Then since for any $k \ge 0$,
\[
	\frac{1}{L_n} \sum_{t = 0}^k tN_t = \frac{1}{L_n} \sum_{t = 0}^k t \sum_{i = 1}^n 
	\ind{D_i = t} = F_n^\ast(k), 
\]
it follows from \eqref{eq:existence_edge_dga} that
\begin{equation}\label{eq:def_psi_n}
	\psi_n(k,\ell) := \ind{1 - F_n^\ast(k) < F_n^\ast(\ell)}\ind{1 - F_n^\ast(k - 1) > 
    	F_n^\ast(\ell - 1)}.
\end{equation}
Moreover, if we let $\mathcal{E}_n(k, \ell)$ denote the average number of edges between nodes of degree $k$ 
and $\ell$, then \eqref{eq:number_edges_dga} implies that
\begin{equation}\label{eq:def_mathcal_E_n}
	\mathcal{E}_n(k, \ell) = \min(1 - F_n^\ast(k - 1), F_n^\ast(\ell)) - \max(1 - F_n^\ast(k), 
    	F_n^\ast(\ell - 1)).
\end{equation}
Summarizing we therefore have the following result.

\begin{proposition}\label{prop:empirical_joint_degree_distribution}
Let $G_n = \emph{\texttt{DGA}}({\bf D}_n)$, for some degree sequence ${\bf D}_n$ and
define the functions $\psi_n$ and $\mathcal{E}_n$, on the positive integers by
\begin{align}
	\psi_n(k,\ell) &= \ind{1 - F_n^\ast(k) < F_n^\ast(\ell)}\ind{1 - F_n^\ast(k - 1) > 
    	F_n^\ast(\ell - 1)} \quad \text{and} \\
    \mathcal{E}_n(k, \ell) &= \min(1 - F_n^\ast(k - 1), F_n^\ast(\ell)) - \max(1 - F_n^\ast(k), 
    	F_n^\ast(\ell - 1)).
\end{align} 
Then,
\[
	h_n(k, \ell) = \psi_n(k,\ell)\mathcal{E}_n(k, \ell).
\]
\end{proposition}

From Proposition \ref{prop:empirical_joint_degree_distribution} we obtain the limiting joint degree distribution as defined in \eqref{eq:joint_degree_density}, when ${\bf D}_n \in 
\mathcal{D}_{\eta, \, \varepsilon}(f,f^\ast)$. 

\begin{theorem}\label{thm:joint_degree_distribution}
Let ${\bf D}_n \in \mathcal{D}_{\eta, \varepsilon}(f,f^\ast)$ and $G_n = \emph{\texttt{DGA}}({\bf D}_n)$. In 
addition let $h(k,ell)$ be as defined in \eqref{eq:joint_degree_density}, take $0 < \delta < \varepsilon$, $K > 0$ and define the event
\[
	\Xi_n = \left\{\sum_{k, \ell = 0}^\infty \left|h_n(k,\ell) - h(k, \ell)\right| \le 
	K n^{-\delta}\right\}.
\]
Then
\[
	\Prob{\Xi_n} \ge 1 - \bigO{n^{-\varepsilon + \delta} + \Prob{\Omega_n^c}},
\]
as $n \to \infty$.
\end{theorem}

We will use Theorem \ref{thm:joint_degree_distribution} in Section \ref{ssec:proof_main_result} to prove our main result, Theorem \ref{thm:main_result}.

\subsection{Properties of the Disassortative Graph Algorithm}

We will now address several properties of the \texttt{Disassortative Graph Algorithm}. The first is
concerned with the optimization problem \eqref{eq:approx_optimization_problem}. 

\begin{theorem}\label{thm:optimality_algorithm}
The \texttt{Dissassortative Graph Algorithm} solves \eqref{eq:approx_optimization_problem}.
\end{theorem}

This result can be explained as follows. Let ${\bf a}_{L_n}$ be the list of degrees with respect to the labels of the stubs, ordered in descending order. That is
\[
	{\bf a}_{2L_n} = \left(\underbrace{D_{\phi(n)}, \, \dots, D_{\phi(n)}}_{D_{\phi(n)}}, 
    D_{\phi(n-1)}, \, \dots \dots, \underbrace{D_{\phi(N_1)}, \, \dots, D_{\phi(1)}}_{N_1}\right).
\]
Then the \texttt{DGA} pairs the degrees $a_i$ and $a_{L_n + 1 - i}$, which minimizes $\sum_{i \to j \in G} \mathcal{F}_n^\ast(D_i)\mathcal{F}_n^\ast(D_j)$ and hence the \texttt{DGA} minimizes Spearman's rho $\widetilde{\rho}(G_n)$. Observe that, in addition, the algorithm minimizes $\sum_{i \to j \in G} D_i D_j$ so that we also obtain the minimum for the $s$ metric of the graph $G$, $s_{min}$, as considered in \cite{Alderson2007}. Moreover, the fact that we could use an arbitrary algorithm to connect the nodes of degree $z_n$ confirms the observation in \cite{Alderson2007} that graphs with minimal $s$ metric are not unique with respect to their structure.

As we have already mentioned, the joint degree structure, and hence the optimality of the 
\texttt{DGA}, depends only on the degree of nodes that are connected and not on their 
labels. In the algorihtm, however, we use an ordering for filling the lists of stubs $S_n$ and 
$T_n$. This is to make sure that the probability 
that $G_n^\ast$ is simple, i.e. it has no self-loops and no more than one edge between nodes $i$ 
and $j$, converges to one as $n \to \infty$.

To understand the intuition behind the proof, consider the first time the algorithm sees a stub belonging to a node $i$ in 
the list $S_n$ with degree $D_i > z_n$. Then node $i$ will be connected to the nodes corresponding 
to the next $D_i$ stubs in $T_n$. Now consider such a stub, belonging to node $j$. Then there will 
be more then one edge $i - j$ if and only if there is more than one stub belonging to node $j$ 
among the $D_i$ stubs in $T_n$, which can only happen when $D_i > N_{D_j}$. Since the degree of 
nodes in $T_n$ is bounded by $z_n$, we have that $N_{D_j}$ scales as $n$, while the maximal degree 
is $o(n)$, since $f$ has finite mean. Therefore, the event $D_i > N_{D_j}$ for $D_i > z_n$ and 
$D_j \le z_n$ has vanishing probability. We hence have the following result, the details of the proof can be found in Section \ref{ssec:proof_G_ast_simple}.

\begin{proposition}\label{prop:G_ast_simple}
Let ${\bf D}_n \in \mathcal{D}_{\eta,\varepsilon}(f, f^\ast)$,  $G_n = \texttt{DGA}({\bf D}_n)$ and denote by $\mathcal{S}_n^\ast$ the event that $G_n^\ast$ is simple, then
\[
	\Prob{\mathcal{S}_n^\ast} \ge 1 - O\left(n^{-\varepsilon} + n^{-1/2} 
    + n^{-\eta/2} + \Prob{\Omega_n^c}\right),
\]
as $n \to \infty$.
\end{proposition}

This proposition implies that the simplicity of the graph $G_n$, generated by the \texttt{Disassortative Graph Algorithm}, solely depends on the simplicity of $G_n^z$, constructed in Step 19. Now consider the degree sequence ${\bf D}_n^z$ corresponding to the remaining stubs, obtained in Step 18 and observe that these degrees are uniformly bounded by $z_n$. Take ${\bf D}_n \in \mathcal{D}_{\eta, \, \varepsilon}(f, f^\ast)$ and let $z$ the be the median of $F^\ast$, i.e. the unique integer such that
\begin{align} \label{eq:relation_z}
	F^\ast(z) \ge \frac{1}{2} \quad \text{and} \quad F^\ast(z - 1) < \frac{1}{2}.
\end{align}
We can show that
\[
	\lim_{n \to \infty} \Prob{z_n \le z + 1} = 1,
\]
see the proof of Proposition \ref{prop:G_ast_simple} in Section \ref{ssec:proof_G_ast_simple}. Therefore, if we define the
event $A_n = \{z_n \le z + 1\}$, then conditioned on $A_n$ the degrees in ${\bf D}_n^z$ are bounded by 
$z + 1$. Hence, if we connect these stubs using the configuration model, and let $\mathcal{S}_n^z$ 
denote the event that $G_n^z$ is simple, then it follow, see e.g. \cite{Hofstad2014a} Theorem 7.12, 
that there exist a constant $\mathfrak{s} > 0$, such that
\[
	\lim_{n \to \infty} \Prob{\mathcal{S}_n} = \lim_{n \to \infty} \Prob{\mathcal{S}_n^z, A_n} + \Prob{A_n^c} = \mathfrak{s}.
\]
From this and Proposition \ref{prop:G_ast_simple} we obtain the following corollary.
\begin{corollary}\label{cor:simplicity_G}
Let ${\bf D}_n \in \mathcal{D}_{\eta, \varepsilon}$, $G_n = \emph{\texttt{DGA}}({\bf D}_n)$ 
and denote by $\mathcal{S}_n$ the event that $G_n$ is simple. Then there exists a constant 
$\mathfrak{s} > 0$ such that
\[
	\lim_{n \to \infty} \Prob{\mathcal{S}_n} = \mathfrak{s}.
\]
\end{corollary}
Note that by Lemma \ref{lem:iid_algorithm} it follows that if $D$ is an
integer valued random variable that satisfies, for some $\eta > 0$,
\[
	\nu := \Exp{D} < \infty \quad \text{and} \quad
	\Exp{D^{1 + \eta}} < \infty,
\]
then a degree sequence ${\bf D}_n$ generated by the \texttt{IID} algorithm satisfies ${\bf D}_n \in \mathcal{D}_{\eta, \varepsilon}$, for any
$0 < \varepsilon < \eta/(8 + 4\eta)$, while
\[
	\Prob{\Omega_n} \ge 1 - \bigO{n^{-\varepsilon}},
\]
as $n \to \infty$. Therefore, if we want to generate maximally disassortative graphs with limit 
degree density $f$, we can first generate a degree sequence using the \texttt{IID} algorithm, by
sampling from $f$, and then connect the nodes using the \texttt{DGA}. From Corollary 
\ref{cor:simplicity_G} it now follows that, in order to generate 
maximally disassortative simple graphs, we could repeat steps 13 to 19 in the 
\texttt{Disassortative Graph Algorithm} until the resulting graph is simple.

\section{Spearman's rho and the tail of the degree distribution}
\label{sec:spearman_tail_distribution}

We will now investigate the influence of the degree distribution on the value of 
Spearman's rho, on maximally disassortative graphs, i.e. graphs generated by the \texttt{DGA}. We 
will show that the tail of the distribution does not influence this value. This is achieved by transforming a 
given degree distribution, such that the asymptotic behavior of the tail of this distribution is
preserved, while we increase the probability mass of the corresponding size-biased degree 
distribution at one.

Let us start by considering a degree distributions $f$, for which the size-biased distributions 
$f^\ast$ satisfies $f^\ast(1) \ge 1/2$, i.e. $f^\ast$ has median $1$. Observe that in this case we 
have
\begin{align*}
	h(k, \ell) = \begin{cases}
    	f^\ast(k) &\mbox{if } k > 1 \text{ and } \ell = 1 \\
       	2f^\ast(1) - 1 &\mbox{if } k = 1 \text{ and } \ell = 1 \\
        f^\ast(\ell) &\mbox{if } k = 1 \text{ and } \ell > 1 \\
        0 &\mbox{else.}
    \end{cases}
\end{align*}
Hence, if $D_\ast, D^\ast$ have joint distribution $h$, as defined in 
\eqref{eq:joint_degree_density}, then
\begin{align*}
	\Exp{\mathcal{F}^\ast(D_\ast)\mathcal{F}^\ast(D^\ast)} 
    &= \sum_{k, \ell = 1}^\infty \mathcal{F}^\ast(k)\mathcal{F}^\ast(\ell)h(k,\ell) \\
    &= f^\ast(1)^2(2f^\ast(1) - 1) + 2f^\ast(1)\sum_{\ell = 2}^\infty \mathcal{F}^\ast(\ell) 
			f^\ast(\ell) \\
    &\ge f^\ast(1)^2(2f^\ast(1) - 1) + 4f^\ast(1)^2\sum_{\ell = 2}^\infty f^\ast(\ell) \\
    &= f^\ast(1)^2(2f^\ast(1) - 1) + 4f^\ast(1)^2(1 - f^\ast(1)) \\
    &= 3f^\ast(1)^2 - 2f^\ast(1)^3,
\end{align*}
where we used, see \eqref{eq:definition_mathcal_F}, that $\mathcal{F}^\ast(\ell) \ge 2f^\ast(1)$ 
for all $\ell > 1$. From this it follows that whenever $f^\ast(1) \ge 1/2$ and $D_\ast, D^\ast$ have joint distribution $h$,
\begin{equation}\label{eq:spearman_lower_bound_a1}
	3\Exp{\mathcal{F}^\ast(D_\ast)\mathcal{F}^\ast(D^\ast)} - 3 \ge 9f^\ast(1)^2 - 6f^\ast(1)^3 - 3.
\end{equation}

Since the function on the right side of \eqref{eq:spearman_lower_bound_a1} is strictly 
monotonically increasing and is $0$ when $f^\ast(1) = 1$, it follows that the limit of Spearman's 
rho on maximally disassortative graphs can be bounded from below by a value that is arbitrary 
close to $0$, if $f^\ast(1)$ is large enough. Moreover, using that the $h$ is the joint degree distribution of graphs with minimal value of $\widetilde{\rho}$, we have the following result.

\begin{proposition}\label{prop:lower_bound_spearman}
Let $f$ and $f^\ast$ be such that $f^\ast(1) \ge \frac{1}{2}$ and $G_n \in 
\mathcal{G}_{\eta, \, \varepsilon}(f,f^\ast)$. Then, for any $0 < \delta < \min(\varepsilon, 1/2)$
and $K > 0$,
\[
	\Prob{\widetilde{\rho}(G_n) \ge 9f^\ast(1)^2 - 6f^\ast(1)^3 - 3 - K n^{-\delta}} \ge 1 - 
	\bigO{n^{-\varepsilon + \kappa} + \Prob{\Omega_n^c}},
\]
as $n \to \infty$, where $\kappa =(\varepsilon+ \delta)/2$.
\end{proposition}

Given $0 < \delta < 1$, we will now describe a construction that transforms any given distribution 
$f$ with support on the positive integers, into a distribution $f_\xi$, with support on the 
positive integers, such that $f_\delta(1) = \delta$ and
\begin{equation}\label{eq:ch6_asymptotic_tails_f_delta}
	\lim_{k \to \infty} \frac{1 - F_\delta(k)}{1 - F(k)} = 1,
\end{equation}
where $F$ and $F_\delta$ are the cumulative distribution functions of $f$ and $f_\xi$, respectively.

\begin{algorithm*}
	\caption{$\delta$-transform of a probability density $f$}
	\label{alg:ch6_delta_transfrom}
	\begin{algorithmic}[1]
		\STATE Given: a probability density $f$, corresponding cdf $F$ and $0 < \delta < 1$
		\STATE Let $K_\delta$ be the smallest integer such that $F(K_\delta) > \delta$.
		\STATE Set $x = F(K_\delta) - \delta > 0$ and $f_\delta(1) = \delta$.
		\IF{$K_\delta = 1$}
			\STATE Set $f_\delta(2) = f(2) + x$
		\ELSE
			\FOR{$2 \le k \le K_\delta$}
			\STATE Set $f_\delta(k) = x/(K_\delta - 1)$
			\ENDFOR
			\STATE Set $f_\delta(K_\xi + 1) = f(K_\delta + 1)$
		\ENDIF
		\FOR{$k > K_\delta$}
			\STATE Set $f_\delta(k) = f(k)$
		\ENDFOR
		\STATE Output: Probability density $f_\delta$
	\end{algorithmic}
\end{algorithm*}
To see that $f_\delta$ defines a probability density function we compute
\begin{align*}
	\sum_{t = 1}^\infty f_\delta(t) &= \delta + \sum_{t = 2}^{K_\delta+1} f_\delta(t) 
		+ \sum_{t = K_\delta + 2}^\infty f(t) \\
   &= \delta + x + f(K_\delta + 1) + 1 - F(K_\delta + 1) = 1.
\end{align*}
Moreover, since $f_\delta(k) = f(k)$ for all $k > K_\delta$ it follows that $F_\delta$ satisfies 
\eqref{eq:ch6_asymptotic_tails_f_delta}. We will refer to $f_\delta$ as the \emph{$\delta$-transform} of $f$. 

With this transformation we can now transform a given distribution $f$, to get a distribution $f_\rho$ 
whose size-biased distribution $f^\ast_\rho$ satisfies
\[
	9f_\rho^\ast(1)^2 - 6f_\rho^\ast(1)^3 - 3 > \rho,
\]
without affecting the asymptotic behavior of the tail of the original distribution $f$. It then follows
from Proposition \ref{prop:lower_bound_spearman} that for any sequence of graphs $G_n \in 
\mathcal{G}_{\eta, \, \varepsilon}(f_\rho, f^\ast_\rho)$,
\[
	\lim_{n \to \infty} \Prob{\widetilde{\rho}(G_n) > \rho} = 1,
\]
which proves Theorem \ref{thm:independence_spearman_tail}. The details can be found in Section 
\ref{ssec:proof_main_result}.

The construction we use for creating the 
adversary degree distribution $f_\rho$ has one downside. In order to construct degree distributions 
such that $\widetilde{\rho}(G_n)$ is arbitrary close to zero, the value of $f^\ast(1)$ should be arbitrary 
close to $1$. Therefore, these distributions might not resemble real-world situations. The reason 
for this downside is that the construction is based on the very crude lower bound 
\eqref{eq:spearman_lower_bound_a1} on Spearman's rho, for which we had to assume $f^\ast(1) \ge 
1/2$. 

As we mentioned in Section \ref{ssec:main_results}, Theorem \ref{thm:independence_spearman_tail} states 
that the minimal value of Spearman's rho and not determined by the tail of the distribution. 

Now let $F$ be regularly varying with exponent $\gamma > 1$ 
and slowly varying function $\mathcal{L}$, see \eqref{eq:regularly_varying_degrees}. Pick any $-1 < \rho < 0$ and let 
$F_\rho$ be the transformed distribution, given by Theorem 
\ref{thm:independence_spearman_tail}. We will show that $F_\rho$ is again regularly varying 
with exponent $\gamma$. Note that for this it is enough to show that
$(1 - F_\rho(x))x^\gamma$ is slowly varying. To this end fix $t > 0$ and write
\begin{align*}
	\frac{1 - F_{\rho}(tk)}{t^{-\gamma}(1 - F_{\rho}(k))} 
    &= \left(\frac{1 - F_{\rho}(tk)}{1 - F(tk)}\right) 
        \left(\frac{1 - F(k)}{1 - F_{\rho}(k)}\right) 
		\left(\frac{1 - F(tk)}{t^{-\gamma}(1 - F(k))}\right) \\
	&= \left(\frac{1 - F_{\rho}(tk)}{1 - F(tk)}\right) 
    	\left(\frac{1 - F(k)}{1 - F_{\rho}(k)}\right) \frac{\mathcal{L}(tk)}{\mathcal{L}(k)}.
\end{align*}
The product of the first two terms converge to $1$, as $k \to \infty$, by Theorem 
\ref{thm:independence_spearman_tail}, while this holds for the last term since 
$\mathcal{L}$ is slowly varying. Summarizing, we have
\[
	\lim_{k \to \infty} \frac{1 - F_{\rho}(tk)}{t^{-\gamma}(1 - F_{\rho}(k))} = 1,
\]  
which proves that $(1 - F_\rho(x))x^\gamma$ is slowly varying and hence $F_\rho$ is 
regularly varying with exponent $\gamma$. This proves Corollary \ref{cor:independence_rho_scale_free}.

\section{Spearman's rho on maximal disassortative graphs.}
\label{sec:spearman_experiments}

We will now use numerical experiments to illustrate the behavior of Spearman's rho for two types of degree 
distributions, regularly varying and Poisson. Each of these types has a parameter that can serve as
a proxy for the way in which the mass of the probability density functions is distributed over 
their support. For the regularly varying distributions this is the exponent $\gamma$, while for the
Poisson distribution it is the mean $\lambda$. We will refer to these as the parameters of the 
distribution. 

For the simulations we generated degree sequences ${\bf D}_n$ by sampling from the given 
distribution, using the \texttt{IID} algorithm, for different sizes $n$ and values for the 
parameters. We then generated graphs $G_n$ using the \texttt{DGA}. For 
each combination of size and parameter, we generated $10^3$ graphs in this manner and computed 
$\rho(G_n)$, as defined in \eqref{eq:spearmans_rho_full}, on each of them. This gives us $10^3$ samples of Spearman's rho on maximal 
disassortative graphs with the given size and degree distribution. 

To analyze the speed of convergence of $\rho(G_n)$ we computed for each combination of 
size and parameter
\[
	X_n := \left|\rho(G_n) - \mathbb{E}^\prime\left[\rho(G_n)\right]
	\right|,
\]
where $\mathbb{E}^\prime$ denotes the empirical mean, based on the $10^3$ realizations per such 
combination. We then plotted the empirical inverse cumulative distribution of $X_n$ for 
different sizes $n = 10^4, 10^5, 10^6$ and $10^7$. The results are shown in Figure 
\ref{fig:concentration_spearman_pareto} and Figure 
\ref{fig:concentration_spearman_poisson}. 

In addition, to investigate the limit of Spearman's rho in maximally disassortative graphs, we 
computed $\mathbb{E}^\prime\left[\rho(G_n)\right]$, with $n = 10^7$, for several values
of the parameter of the distribution. We then plotted these values with respect to the parameter
in Figure \ref{fig:dga_spearman}.

We will now describe the specific distributions we used for the simulations and discuss the 
results.

\subsection{Scale-free degree distribution}

Let $X$ have a Pareto distribution with scale $1$ and shape $\gamma > 1$, i.e.
\[
	f_X(t) = \begin{cases}
		\gamma t^{-1-\gamma} &\mbox{if } t \ge 1 \\
        0 &\mbox{else},
	\end{cases}
    \qquad
    1 - F_X(t) = \begin{cases}
    	t^{-\gamma} &\mbox{if } t \ge 1 \\
        1 &\mbox{else},
    \end{cases}
\]
and define $D = \floor{X}$. Then we have that $1 - F(k) = 1 - F_X(k + 1)$, so that $F$ is regularly
varying with exponent $\gamma > 1$, while
\begin{equation}\label{eq:pareto_degree_distribution}
	f(k) = F(k) - F(k-1) = k^{-\gamma} - (k + 1)^{-\gamma}.
\end{equation}
Standard calculations yield that $\sum_{k = 0}^\infty k f(k) = \zeta(\gamma)$, where $\zeta$ is the 
Riemann zeta function. Therefore we have that
\[
	f^\ast(k) = \frac{k f(k)}{\zeta(\gamma)},
\]
so that $f^\ast(1) = (1 - 2^{-\gamma})/\zeta(\gamma)$ which is increasing in $\gamma$. Moreover 
$9f^\ast(1)^2 - 6f^\ast(1)^3 - 3 > -1$ for all $\gamma \ge 2.5$, which places it in the class of 
adversary distributions we considered in the previous section.

From Figure \ref{fig:concentration_spearman_pareto} we see that $\rho(G_n)$ is already strongly
concentrated around it's mean when $n = 10^5$. Even when we the degree distribution has infinite 
variance $(\gamma = 1.5)$ we have that $X_n \le 0.025$, with high probability, for graphs 
of size $n = 10^5$. This shows, complementary to Theorem \ref{thm:main_result}, 
that the \texttt{DGA} performs very well in practice with respect to the convergence of Spearman's
rho to the minimal achievable value $\rho(D_\ast, D^\ast)$.
 
Interestingly, the simulations suggest that the concentration of $\rho(G_n)$ around its mean for 
graphs of small size becomes tighter when $\gamma$ decreases. Compare, for instance, the plots 
for $n = 10^4$ in the Figure \ref{sfig:spearman_pareto_15} - \ref{sfig:spearman_pareto_25}.

In Figure \ref{sfig:spearman_pareto} we plotted the empirical average of 
$\rho(G_n)$ against the parameter $\gamma$ of the degree density 
\eqref{eq:pareto_degree_distribution}. Observe that in contrary to the lower bound related to 
$f^\ast(1)$, we clearly see that Spearman's rho is strongly increasing as a function of $\gamma$ 
and $\rho(G_n) > -0.8$ for $\gamma > 2$. 
Therefore it follows by Theorem \ref{thm:main_result} that the rank-correlation measure Spearman's rho on any graph with 
degree distribution \eqref{eq:pareto_degree_distribution} and $\gamma > 2$ will not have a 
value smaller than $-0.8$. Moreover, when $\gamma \ge 2.5$ we see that 
$\mathbb{E}^\prime[\rho(G_n)] >-0.5$. Since this is a lower bound for Spearman's rho on any graph
with degree density \eqref{eq:pareto_degree_distribution}, a consequence could be that even 
if such graphs have a very disassortative joint degree structure they could potentially be 
classified differently.

\begin{figure}[!p]
	\centering
	\begin{subfigure}{0.3\linewidth}
		\centering
		\includegraphics[scale=0.33]{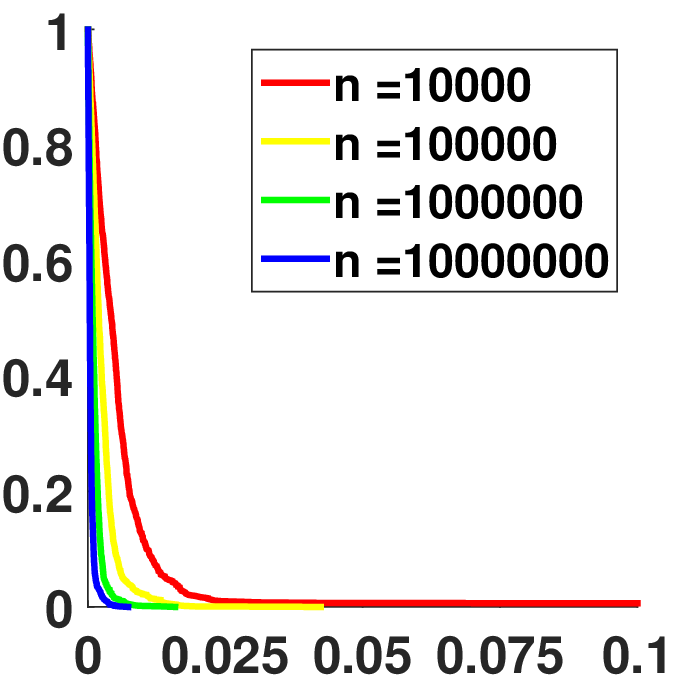}
		\caption{$\gamma = 1.5$}
		\label{sfig:spearman_pareto_15}
	\end{subfigure}~
    \begin{subfigure}{0.3\linewidth}
		\centering
		\includegraphics[scale=0.33]{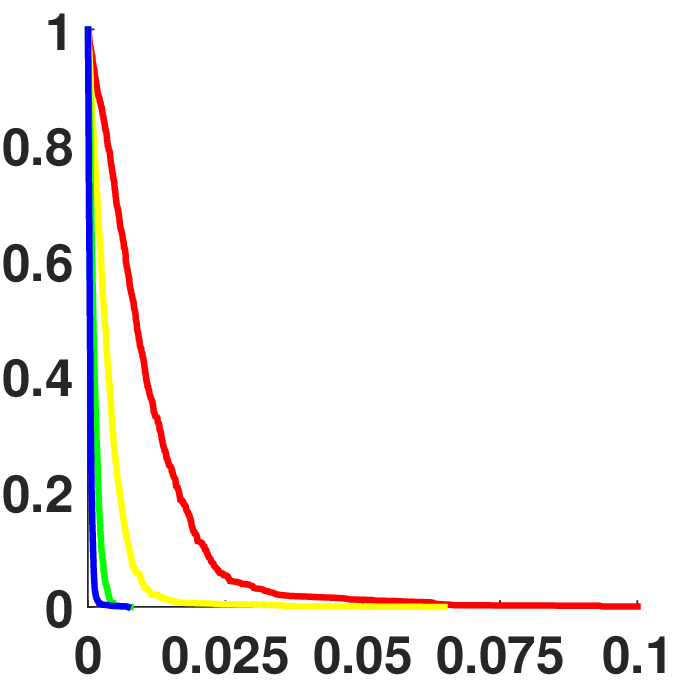}
		\caption{$\gamma = 2$}
		\label{sfig:spearman_pareto_20}
	\end{subfigure}
   \begin{subfigure}{0.3\linewidth}
		\centering
		\includegraphics[scale=0.33]{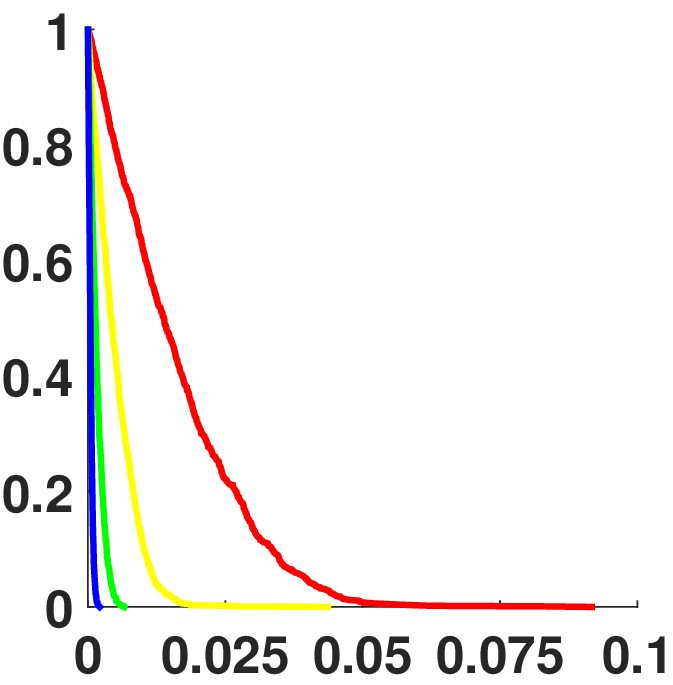}
		\caption{$\gamma = 2.5$}
		\label{sfig:spearman_pareto_25}
	\end{subfigure}
	\caption{Plot of the inverse cdf of $X_n$ for graphs of different sizes and degree     
    distribution \eqref{eq:pareto_degree_distribution}, generated by the \texttt{DGA}, for three different 
    choices of $\gamma$.}
	\label{fig:concentration_spearman_pareto}
\end{figure}

\begin{figure}[!p]
	\centering
	\begin{subfigure}{0.3\linewidth}
		\centering
		\includegraphics[scale=0.33]{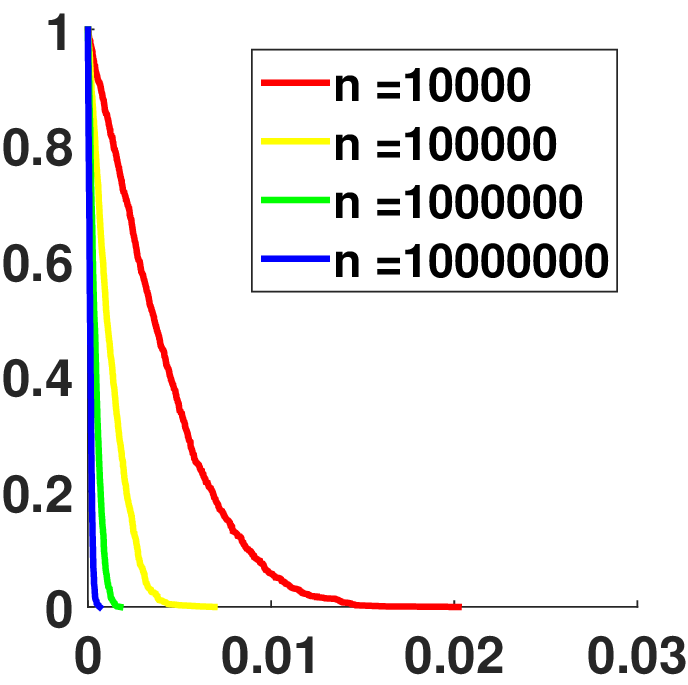}
		\caption{$\lambda = 0.5$}
	\end{subfigure}~
  \begin{subfigure}{0.3\linewidth}
		\centering
		\includegraphics[scale=0.33]{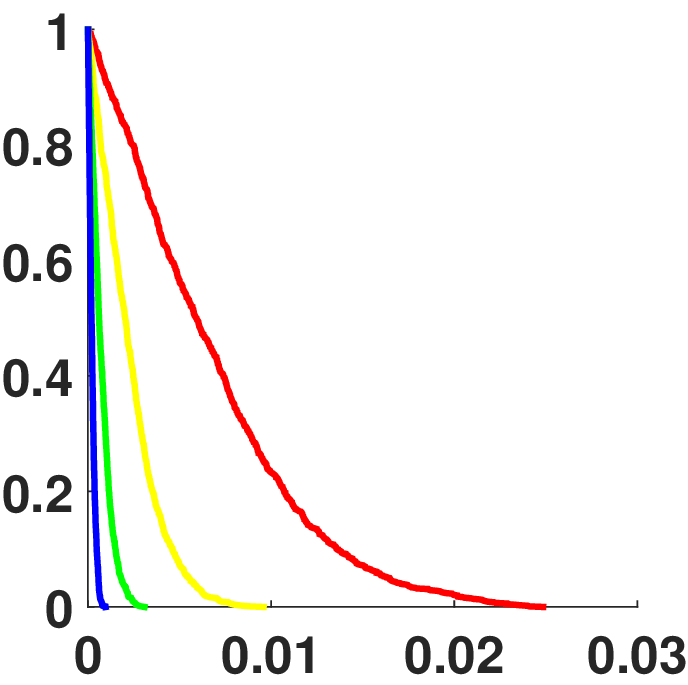}
		\caption{$\lambda = 1.0$}
	\end{subfigure}
  \begin{subfigure}{0.3\linewidth}
		\centering
		\includegraphics[scale=0.33]{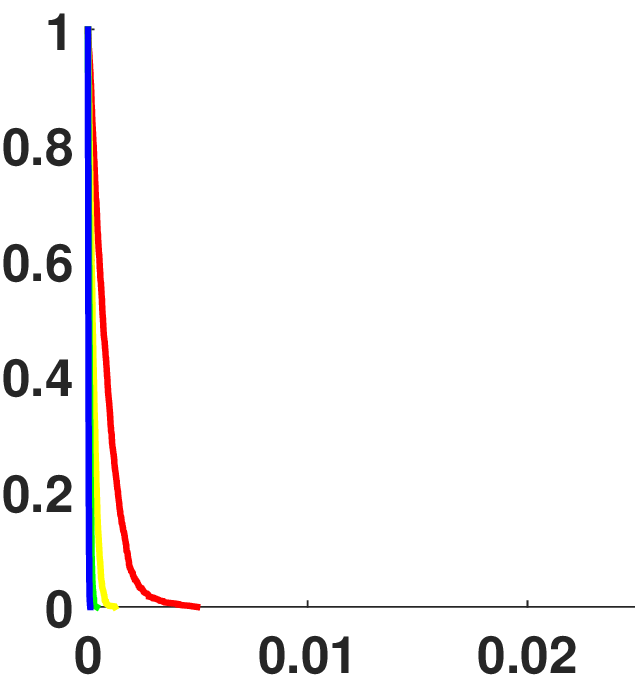}
		\caption{$\lambda = 5.0$}
	\end{subfigure}
    \caption{Plot of the inverse cdf of $X_n$ for graphs of different sizes and Poisson degree     
    distribution, generated by the \texttt{DGA}, for three different 
    choices of $\lambda$.}
	\label{fig:concentration_spearman_poisson}
\end{figure}

\begin{figure}[!p]
	\centering
	\begin{subfigure}{0.48\linewidth}
		\centering
		\includegraphics[scale=0.33]{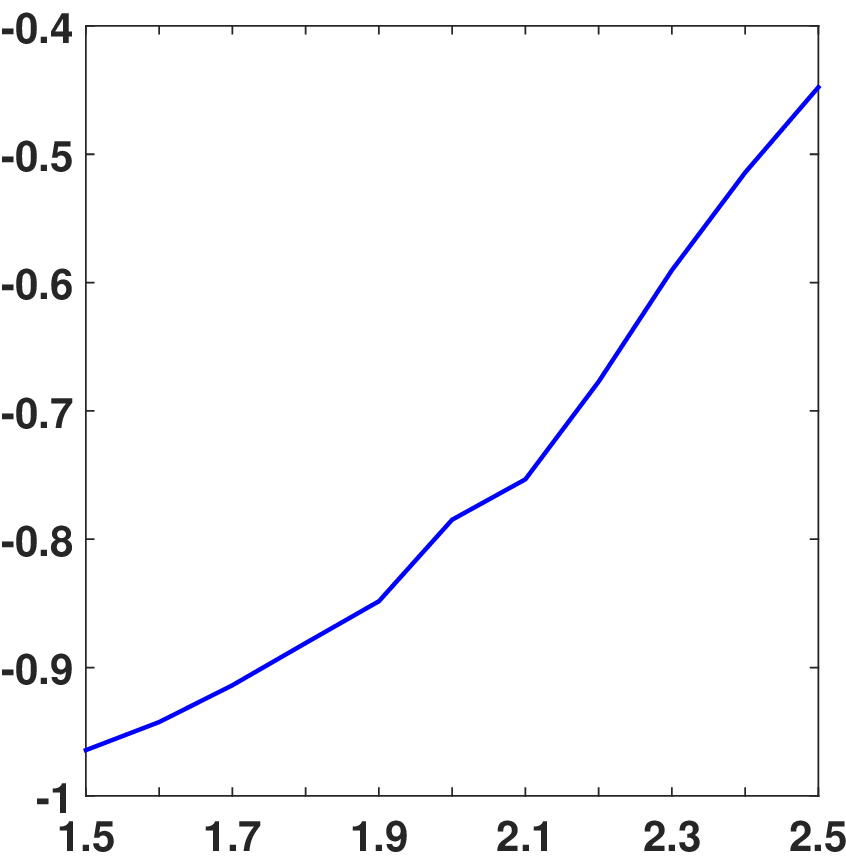}
		\caption{Pareto}
		\label{sfig:spearman_pareto}
	\end{subfigure}~
	\begin{subfigure}{0.48\linewidth}
		\centering
		\includegraphics[scale=0.33]{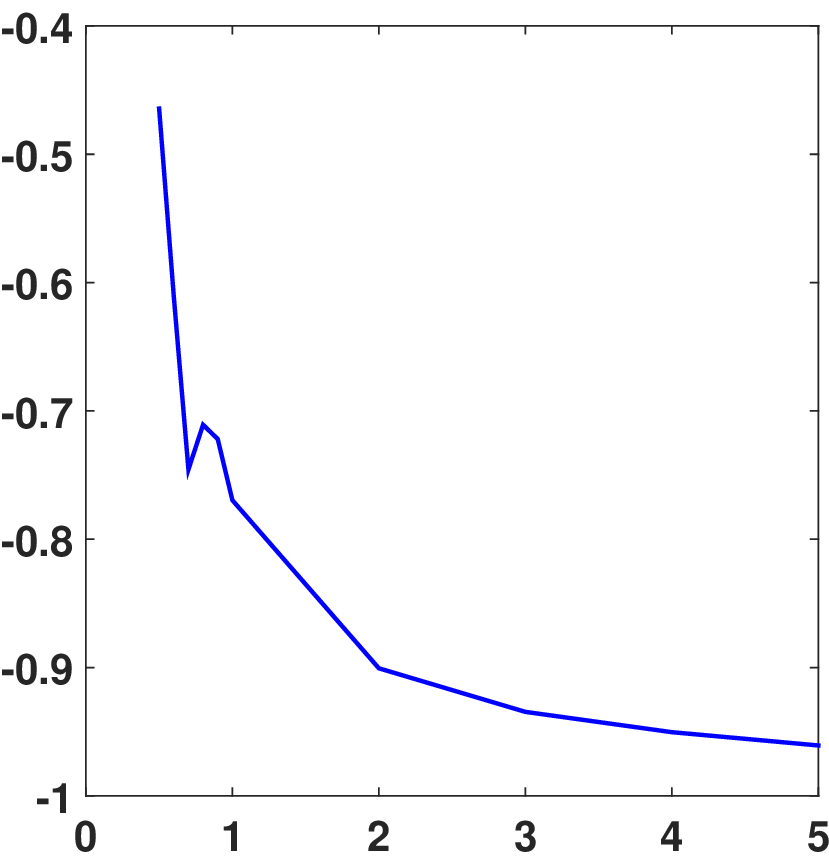}
		\caption{Poisson}
		\label{sfig:spearman_poisson}
	\end{subfigure}
	\caption{Plot of the empirical average of $\rho(G_n)$ for graph of size $10^7$ and 
		degree distributions \eqref{eq:pareto_degree_distribution} and Poisson, generated by the 
		\texttt{DGA}, for different values of, respectively, $\gamma$ and $\lambda$.}
	\label{fig:dga_spearman}
\end{figure}

\subsection{Poisson degree distribution}

Let $X$ be a Poisson random variable with mean $\lambda$ and denote its probability by $f$. Then it
follows that $f^\ast(k) = f(k-1)$. Hence $f^\ast(1) = e^{-\lambda}$ is a decreasing function of 
$\lambda$ and $9f^\ast(1)^2 - 6f^\ast(1)^3 - 3 > -1$ for at least all $\lambda \le 0.4$. This is
opposite to the degree distribution \eqref{eq:pareto_degree_distribution}, where $f^\ast(1)$ was
an increasing function of the parameter $\gamma$. This is reflected in Figure 
\ref{sfig:spearman_poisson}, where we see that $\mathbb{E}^\prime[\rho(G_n)]$ decreases with 
$\lambda$. Here we again see that the shape of the degree distribution strongly influences the 
value of $\rho(G_n)$ for maximally disassortative graphs, and hence the minimal value that Spearman's rho
can attain for any graph with this degree distribution. Note that, in contrast to the case with
the regularly varying distribution, $\rho(G_n)$ is not monotonic with respect to $\lambda$. This 
could be due to the fact that the Poisson density is non-monotonic, while the density
\eqref{eq:pareto_degree_distribution} is monotonically decreasing.

In addition we also observe that, similar to the previous setting, the \texttt{DGA} performs very well 
with respect to the convergence of $\rho(G_n)$. Already for very reasonable sizes, $n \ge 10^5$, 
the deviations around the mean are, with high probability, smaller then $0.02$ for all three values
of $\lambda$.

\subsection{Important observations and insights}\label{ssec:insights}

The main observation from the simulations that we did is that the  distribution of the mass of 
the degree 
probability density is of crucial importance for the minimal value that Spearman's rho can attain.
Moreover, it seems that already for very reasonable degree distributions this minimum is much
larger than $-1$. Therefore, one should be careful when classifying a network as not being very
disassortative when a small negative value of Spearman's rho is computed.

The simulations suggest something even stronger. For this, consider the probability density 
\eqref{eq:pareto_degree_distribution} and observe that if we increase
$\gamma$ than the atoms at the end of this density lose mass, while those at the beginning gain 
mass. In this way we can use the parameter $\gamma$ to 'shift' mass between the head and tail of 
the distribution. The mean, $\lambda$, of the Poisson can be used in a similar way, although in 
this case we need to decrease $\lambda$ in order to move mass towards the head. For both distributions we see that, as the mass of the probability density 
$f$ is moved towards the tail (decreasing $\gamma/$increasing $\lambda$), the value of $\rho(G_n)$ in maximally disassortative graphs with this degree distribution decreases and seems to approach $-1$. On the other hand, as we move more mass to the head of the probability density $f$ (increasing $\gamma/$decreasing $\lambda$) the minimal value of Spearman's rho increases and seems to go to zero.

\section{Proofs} \label{sec:proofs}

Here we prove the results stated in this paper. 

\subsection{Generating degree sequences ${\bf D}_n \in \mathcal{D}_{\eta, \, \varepsilon}(f,f^\ast)$}

\begin{proof}[Proof of Lemma \ref{lem:iid_algorithm}]
We remark that altering the last degree by at most $1$, to make the sum even, constitutes a correction term of order $n^{-1}$. Hence we will consider the degrees $D_i$ as i.i.d. samples from $D$. 

Now fix $\varepsilon \le \eta/(8 + 4\eta)$ and define the events
\begin{align*}
	A_n &= \left\{\sum_{k = 0}^\infty\left|\sum_{t = 0}^k f_n(t) - f(t)\right| 
    	\le n^{-\eta/(2 + 2\eta)}\right\} \\
    B_n &= \left\{\sum_{k = 0}^\infty\left|f_n^\ast(k) - f^\ast(k)\right| \le n^{-\varepsilon} \right\}
\end{align*}
and notice that $\Prob{\Omega_n^c} \le \Prob{A_n^c} + \Prob{A_n \cap B_n^c}$. For the first term we
have, using Markov's inequality,
\[
	\Prob{A_n^c} \le n^{\frac{\eta}{2 +2\eta}}\Exp{d_1(f_n, f)} \le \bigO{n^{-\frac{\eta}{2 + 2\eta}}}
	= \bigO{n^{-\varepsilon}},
\]
as $n \to \infty$, where the second inequality follows from \cite[Proposition 4.2]{Chen2015}
and the last since $-\eta/(2 + 2\eta) < -\eta/(8 + 4\eta)$. Hence, we need to show that, as $n \to 
\infty$
\[
	\Prob{A_n \cap B_n^c} \le \bigO{n^{-\varepsilon}}.
\]
For this we compute that,
\begin{align*}
	\sum_{k = 0}^\infty\left|f_n^\ast(k) - f^\ast(k)\right| &= \sum_{k = 0}^\infty
	\left|\frac{1}{L_n} \sum_{i = 1}^n D_i \ind{D_i = k} 
	- \frac{\Exp{D \ind{D = k}}}{\mu}\right| \\
	&\le \sum_{k = 0}^\infty \left|\frac{1}{L_n} - \frac{1}{\mu n}\right|\sum_{i = 1}^n D_i 
		\ind{D_i = k} \\
	&\hspace{10pt}+ \sum_{k = 0}^\infty \frac{1}{\mu n}\left|\sum_{i = 1}^n D_i \ind{D_i = k} 
		- \Exp{D \ind{D = k}}\right| \\
	&= \frac{|L_n - \mu n|}{\mu n} + \sum_{k = 0}^\infty \frac{1}{\mu n}\left|\sum_{i = 1}^n 
		X_{ik}\right|,
\end{align*}
where we defined $X_{i k} = D_i\ind{D_i = k} - \Exp{D \ind{D = k}}$. Now observe that 
conditioned on $A_n$ we have that 
\begin{align*}
	\frac{|L_n - \mu n|}{\mu n^{1 - \varepsilon}} 
	&\le \frac{n^{1 - \eta/(2 + 2\eta)}}{\mu n^{1 - \varepsilon}}
		\le\mu^{-1} n^{\varepsilon - \eta/2(1 + \eta)} \\
	&\le \mu^{-1} n^{-\eta/4(1 + \eta)} \le \mu^{-1} n^{-\eta/4(1 + 2\eta)},
\end{align*}
so that
\[
	\Prob{A_n \cap B_n^c} \le \bigO{n^{-\varepsilon}}
    + \Prob{\sum_{k = 0}^\infty \frac{1}{\mu n}\left|\sum_{i = 1}^n X_{ik}\right| > 
    n^{-\varepsilon}},
\]
as $n \to \infty$.

To analyze the last probability take $a_n = \floor{n^{2\varepsilon/\eta}}$. Then,
\begin{align*}
	\Prob{\sum_{k = 0}^\infty \frac{1}{\nu n}\left|\sum_{i = 1}^n X_{ik}\right| > n^{-\varepsilon}}
    &\le \frac{1}{\mu n^{1 - \varepsilon}} \sum_{k = 0}^{a_n} \Exp{\left|\sum_{i = 1}^n 
    	X_{ik}\right|^2}^{1/2} \\
		&\hspace{10pt}+ \frac{1}{\mu n^{1 - \varepsilon}} \sum_{k = a_n + 1}^\infty 
        \Exp{\left|\sum_{i = 1}^n X_{ik}\right|} \\
    &\le \frac{1}{\mu n^{1/2 - \varepsilon}} \sum_{k = 0}^{a_n} \text{Var}(X_{1k})^{1/2} 
    	+  \frac{n^{\varepsilon}}{\mu} \sum_{k = a_n + 1}^\infty \Exp{|X_{1k}|} \\
    &\le \frac{1}{\mu n^{1/2 - \varepsilon}} \sum_{k = 0}^{a_n} k  +  \frac{2n^{\varepsilon}}{\mu} 
    	\sum_{k = a_n + 1}^\infty \Exp{D \ind{D = k}} \\
    &\le \frac{a_n(a_n + 1)}{2\mu n^{1/2 - \varepsilon}} + \frac{2n^{\varepsilon}}{\mu} 
			\Exp{D \ind{D > a_n}} \\
    &\le \frac{a_n(a_n + 1)}{2\mu n^{1/2 - \varepsilon}} + 2n^{\varepsilon}a_n^{-\eta} \\
    &= \bigO{n^{-\varepsilon}}, 
\end{align*}
as $n \to \infty$. Here, for the last line, we used that
\[
	\frac{4\varepsilon}{\eta} - \frac{1}{2} + \varepsilon \le  \frac{1}{2 + \eta} - \frac{1}{2} 
	+ \frac{\eta}{8 + 4\eta} = -\frac{\eta}{8 + 4\eta} \le -\varepsilon,
\]
so that
\[
	n^{\varepsilon - \frac{1}{2}} a_n^2 = \bigO{n^{\frac{4\varepsilon}{\eta} - \frac{1}{2} + 
	\varepsilon}} = \bigO{n^{-\varepsilon}},
\]
as $n \to \infty$.
\end{proof}

\subsection{Optimality of \emph{DGA}}\label{ssec:proof_optimality_dga}

Theorem \ref{thm:optimality_algorithm} is a consequence of the following lemma.

\begin{lemma}\label{lem:permutation_optimization}
Consider a sequence $0 \le a_1 \leq  \ldots \leq a_m$ and let $\mathcal{P}_m$ denote the set of 
permutations of $\{1, \dots, m\}$. Then
\[
\min_{\sigma \in \mathcal{P}_m}  \sum_k a_k a_{\sigma(k)}  = \sum_k  a_k a_{m-k+1}
\]
and this minimum is achievable for a permutation $\sigma$ if and only if 
\[
	a_{\sigma(1)} \ge a_{\sigma(2)} \geq \ldots  \geq a_{\sigma(n) }.
\]
\end{lemma}

\begin{proof}
\hfill \\
\noindent $[\Rightarrow]$
If  $a_{\sigma(1) } \geq \ldots  \geq a_{\sigma(n) }  $ then $\sum_k a_k a_{\sigma(k)}  = 
\sum_k  a_k a_{m-k+1}$ \par \smallskip

\noindent $[\Leftarrow]$ 
Assume that $\sigma = \arg \min_{\sigma \in S_m}  \sum_k a_k a_{\sigma(k)}$ but 
there exist $a_i < a_j$ such that $a_{\sigma(i)} < a_{\sigma(j)}$. Consider
$\sigma^\ast = \sigma \cdot (i j)$  then $\sum_k a_k a_{\sigma(k)} - \sum_k a_k a_{\sigma^\ast(k)} =  
(a_i - a_j)(a_{\sigma(i)} - a_{\sigma(j)}) > 0$ which contradicts the initial assumption.
\end{proof}

\begin{proof}[Proof of Theorem \ref{thm:optimality_algorithm}]

Consider a degree sequence ${\bf D}_n$, rank it in ascending order and let $\phi(k)$ denotes 
the node with rank $k$ among this degree sequence, as defined in the description of the 
\texttt{DGA}. Now define the sequence ${\bf a}_{L_n}$ by
\begin{equation}\label{eq:mathcalF_sequence_ak}
	a_k = \mathcal{F}_n^\ast(D_{\phi(i)}) \quad \text{for all } \sum_{t = 1}^{i - 1} 
	D_{\phi(t)} < k \le \sum_{t = 1}^i D_{\phi(t)},
\end{equation}
where we use the convention that $\sum_{t = 1}^0 D_{\phi(t)} = 0$.
With this definition, the sequence ${\bf a}_{L_n}$ looks as follows
\begin{align*}
	&\hspace{-20pt}a_1 \leq \ldots \leq a_k \leq \ldots =  \\
	&\underbrace{\mathcal{F}_n^\ast\left(D_{\phi(1)}\right) \leq \ldots \leq  \mathcal{F}_n^\ast
		\left(D_{\phi(1)}\right)}_{D_{\phi(1)}} \leq  \underbrace{\mathcal{F}_n^\ast \left(
		D_{\phi(2)}\right) \leq \ldots \leq  \mathcal{F}_n^\ast \left(
		D_{\phi(2)}\right)}_{D_{\phi(2)}} \leq \ldots.
\end{align*}

Next, we note that for each graph $G \in \mathcal{G}({\bf D}_n)$ there exits a permutation 
$\sigma_G$ of such that  
\[
	\sum_{i \to j \in G} \mathcal{F}_n^\ast(D_i)\mathcal{F}_n^\ast(D_j) 
    = \sum_{k = 1}^{L_n}  a_k a_{\sigma_G(k)}.
\]
Any directed graph, has a corresponding permutation $\sigma$ of $\{1, \dots, 
L_n\}$ which defines how the outbound and inbound stubs of the bi-degree sequence are paired to obtain the graph. However, not every such permutation 
corresponds to a graph which is the bi-directed version of an undirected graph, i.e. for each edge 
$i \to j$ there is exactly one edge $j \to i$. Therefore let $\mathcal{P}({\bf D}_n)$
denote the set of all permutations of $\{1, \dots, L_n\}$ which do corresponds to an undirected graph,
in its directed representation. Then the optimization problem
\eqref{eq:approx_optimization_problem} is equivalent to the following problem
\begin{equation}\label{eq:permutation_optimization_graphs}
	\min_{\sigma \in \mathcal{P}({\bf D}_n)}\sum_{k = 1}^{L_n}  a_k a_{\sigma^\ast(k)}.
\end{equation}

Now, recall the partitioned representation of the \texttt{DGA} we introduced in Section 
\ref{ssec:joint_degree_distribution}, see Figure \ref{fig:edge_degree_structure_dga}. From 
this description of the algorithm it is not hard to see that, if ${\bf a}_{L_n}$ is defined by 
\eqref{eq:mathcalF_sequence_ak}, then there exists a permutation $\sigma^\ast$ with the 
property that 
\[
	a_{\sigma^\ast(1)} \ge a_{\sigma^\ast(2)} \ge \dots \ge a_{\sigma^\ast(L_n)},
\]
such that the \texttt{DGA} pairs the stubs corresponding to $a_i$ and $a_{\sigma^\ast(i)}$. 
Therefore, Lemma 
\ref{lem:permutation_optimization} implies that
\[
	\sum_{k = 1}^{L_n} a_k a_{\sigma^\ast(k)} 
	= \min_{\sigma \in \mathcal{P}_{L_n}} \sum_{k = 1}^{L_n} a_k a_{\sigma(k)},
\]
where $\mathcal{P}_{L_n}$ denotes the set of all permutations of $\{1, \dots, L_n\}$. Since
$\mathcal{P}({\bf D}_n) \subseteq \mathcal{P}_{L_n}$, this implies that 
\[
	\sum_{k = 1}^{L_n} a_k a_{\sigma^\ast(k)} 
	= \min_{\sigma \in \mathcal{P}({\bf D}_n)}\sum_{k = 1}^{L_n}  a_k a_{\sigma^\ast(k)},
\]
which proves that the \texttt{DGA} solves \eqref{eq:permutation_optimization_graphs} and hence
it solves \eqref{eq:approx_optimization_problem}
\end{proof}

\subsection{Simplicity of $G_n^\ast$}\label{ssec:proof_G_ast_simple}

\begin{proof}[Proof of Proposition \ref{prop:G_ast_simple}]
Let $z_n$ and $z$ be defined as in \eqref{eq:relation_zn} and \eqref{eq:relation_z}, 
respectively, and define the event
\begin{align*}
	A_n &= \left\{z_n \le z + 1\right\}.
\end{align*}
Then, by definition of $z_n$, we have that $F^\ast(z + 1) > 1/2$ and hence
\begin{align*}
	\Prob{z_n > z + 1, \Omega_n} &\le \Prob{F_n^\ast(z+1) < \frac{1}{2}, \Omega_n} \\
	&\le \Prob{F^\ast(z + 1) - \frac{1}{2} < |F^\ast(z+1) - F^\ast_n(z + 1)|, \Omega} \\
	&\le \frac{\Exp{|F^\ast(z+1) - F^\ast_n(z + 1)|\ind{\Omega_n}}}{F^\ast(z + 1) - \frac{1}{2}} \\
	&\le \bigO{n^{-\varepsilon}},
\end{align*}
as $n \to \infty$.
Therefore, if we define $\Lambda_n = A_n \cap \Omega_n$, it is enough to show that
\[
	1 - \Prob{\mathcal{S}_n, \Lambda_n} \le O\left(n^{-\varepsilon} + n^{-1/2}
    + n^{-\eta/2}\right). 
\]
In order to analyze this probability, note that by construction there are no self-loops in 
$G_n^\ast$. Moreover, a node $i$ with $D_i < z_n$ can only have more than one edge to a node $j$ 
when $D_j > N_{D_i}$. Hence, when $G_n^\ast$ is not simple it means that for some $1 \le k \le z_n$
we must have that $0 < N_k < \max_{1 \le j \le n} D_j$, hence
\[
	\left(\mathcal{S}_n^\ast\right)^c \subseteq \bigcup_{k = 1}^{z_n} 
    \left\{0 < N_k < \max_{1 \le j \le n} D_j\right\}.
\]
Therefore, if we denote $f_{min} = \min_{1 \le k \le z} f(k) > 0$, it follows from the union bound that
\begin{align*}
	 1 - \Prob{\mathcal{S}_n, \Lambda_n} 
     &\le \sum_{k = 1}^{z + 1}\Prob{0 < N_k < \max_{1 \le j \le n} D_j, \Lambda_n} \\
    &= \sum_{k = 1}^{z + 1}\Prob{0 < f_n(k) 
    	< \frac{\max_{1 \le j \le n} D_j}{n}, \Lambda_n} \\
    &\le \sum_{k = 1}^{z + 1}\Prob{f(k) - n^{-\varepsilon} < \frac{\max_{1 \le j \le n} D_j}{n}, 
			\Lambda_n} \\
		&\le \sum_{k = 1}^{z + 1} \Prob{f_{min} < \frac{\max_{1 \le j \le n} D_j}{n} + n^{-\varepsilon}, 
			\Lambda_n} \\
		&\le \frac{(z+1)n^{-1}\Exp{\max_{1 \le j \le n} D_j \ind{\Lambda_n}} 
			+ (z+1)n^{-\varepsilon}}{f_{min}} \\
    &\le \frac{(z + 1)\Exp{\max_{1 \le j \le n} D_j 
    	\ind{D_j > \sqrt{n}} \ind{\Lambda_n}}}{n f_{min}} \\
    &\hspace{10pt}+ \frac{(z + 1)n^{-1/2}}{f_{min}} + \frac{(z+1)n^{-\varepsilon}}{f_{min}}.\\
\end{align*}
The last probability is $O(n^{-1/2} + n^{-\varepsilon})$, as $n \to \infty$. We will now show that 
the other probability is $O(n^{-\varepsilon} + n^{-\eta/2})$. For this we note that
\[
	\frac{\max_{1 \le j \le n} D_j \ind{D_j > \sqrt{n}}}{n} \le \frac{1}{n}\sum_{i = 1}^n D_i 
	\ind{D_i > \sqrt{n}} = \frac{L_n}{n}(1 - F_n^\ast(\sqrt{n})),
\]
and 
\[
	1 - F^\ast(\sqrt{n}) = \Exp{D \ind{D > \sqrt{n}}} \le n^{-\eta/2} \Exp{D^{1 + \eta}}.
\]
Therefore we obtain that, as $n \to \infty$,
\begin{align*}
	\frac{\Exp{\max_{1 \le j \le n} D_j \ind{\Lambda_n}}}{n f_{min}}
	&\le \frac{\Exp{L_n\left|F_n^\ast(\sqrt{n}) - 
		F^\ast(\sqrt{n})\right|\ind{\Lambda_n}}}{n f_{min}} \\
	&\hspace{10pt}+ \frac{\Exp{2L_n(1 - F^\ast(\sqrt{n})\ind{\Lambda_n}}}{n f_{min}} \\
	&\le \frac{\Exp{(\nu n + n^{1-\varepsilon})\sup_{k \ge 0}\left|F_n^\ast(k) - 
		F^\ast(k)\right|\ind{\Lambda_n}}}{n f_{min} } \\
	&\hspace{10pt}+ \frac{\Exp{2(\nu n + n^{1-\varepsilon})(1 - F^\ast(\sqrt{n})\ind{\Lambda_n}}}
		{n f_{min}} \\
	&\le \frac{(\nu + n^{-\varepsilon})n^{-\varepsilon}}{n f_{min}} 
		+ \frac{2(\nu + n^{-\varepsilon})n^{-\eta/2}\Exp{D^{1 + \eta}}}{f_{min}} \\
	&\le O\left(n^{-\varepsilon} + n^{-\eta/2}\right),
\end{align*}
which completes the proof.
\end{proof}

\subsection{Joint degree distribution}\label{ssec:proof_joint_degree_distribution}

Here we will address the convergence of the empirical joint degree density $h_n(k, \ell)$, as defined in Proposition \ref{prop:empirical_joint_degree_distribution}, to the density $h(k, \ell)$ as defined in \eqref{eq:joint_degree_density}.

We will use two technical lemmas, which deal with the difference between the functions $\psi_n$
and $\psi$, and $\mathcal{E}_n$ and $\mathcal{E}$. 

\begin{lemma}\label{lem:technical_psi_n_vs_psi}
Let ${\bf D}_n \in \mathcal{D}_{\eta, \, \varepsilon}(f, f^\ast)$. Then, for any $k, \ell \ge 0$, $0 < 
\delta < \varepsilon$ and $K > 0$
\[
	\Prob{|\psi_n(k,\ell) - \psi(k,\ell)|\mathcal{E}_n(k,\ell) > Kn^{-\delta}, \Omega_n} 
    \le \bigO{n^{-\varepsilon + \delta}}.
\]
\end{lemma}

\begin{lemma}\label{lem:technical_mathcalE_n_vs_mathcalE}
Let ${\bf D}_n \in \mathcal{D}_{\eta, \, \varepsilon}(f, f^\ast)$. Then, for any $k, \ell \ge 0$, $K > 0$
and $0 < \delta < \varepsilon$,
\[
	\Prob{\left|\mathcal{E}(k,\ell) - \mathcal{E}_n(k,\ell)\right| > Cn^{-\delta}, \Omega_n}
    \le \bigO{n^{-\varepsilon + \delta}}.
\]
\end{lemma}

The proof of both lemmas is postponed till the end of this section. We will first give the proof
of Theorem \ref{thm:joint_degree_distribution}.

\begin{proof}[Proof of Theorem \ref{thm:joint_degree_distribution}]
Let $p$ be the smallest integer satisfying 
\begin{equation} \label{eq:thm_joint_degree_tail_atom}
	1 - F^\ast(p) < f^\ast(1),
\end{equation}
and define $p_n$ as the smallest integer that satisfies
\[
	1 - F^\ast_n(p_n) < f_n^\ast(1).
\]
Then we have that $\Prob{p_n = p}$ converges to one, since
\begin{align*}
	&\hspace{-30pt}\Prob{p \ne p_n, \Omega_n} \\
	&= \Prob{1 - F_n^\ast(p) \ge f_n^\ast(1), \Omega_n} \\
	&\le \Prob{(F^\ast(p) - F_n^\ast(p)) +(f^\ast(1) - f_n^\ast(1)) > f^\ast(1) - 1 + F^\ast(p), 
			\Omega_n} \\
	&\le \Prob{|F^\ast(p) - F_n^\ast(p)| > f^\ast(1) - 1 + F^\ast(p),\Omega_n)} \\
	&\hspace{10pt}+ \Prob{|f^\ast(1) - f_n^\ast(1)| > f^\ast(1) - 1 + F^\ast(p),\Omega_n)} \\
	&\le \frac{\Exp{|F^\ast(p) - F_n^\ast(p)|\ind{\Omega_n}}}{f^\ast(1) - 1 + F^\ast(p)} 
		+ \frac{\Exp{|f^\ast(1) - f_n^\ast(1)|\ind{\Omega_n}}}{f^\ast(1) - 1 + F^\ast(p)} \\
  &\le \frac{2n^{-\varepsilon}}{f^\ast(1) - 1 + F^\ast(p)} \le O\left(n^{-\varepsilon}\right),
\end{align*}
as $n \to \infty$, where we used that by definition of $p$ it holds that $f^\ast(1) - 1 + 
F^\ast(p) > 0$.
Therefore, if we define the event $P_n = \{p = p_n\}$ and let $\Lambda_n = P_n \cap \Omega_n$, then 
\[
	\Prob{\Lambda_n} \ge 1 - O\left(n^{-\varepsilon} + \Prob{\Omega_n^c}\right),
\]
so that for Theorem \ref{thm:joint_degree_distribution} it is enough to show that
\begin{equation}\label{eq:thm_joint_degrees_main}
	\Prob{\Xi_n^c, \Lambda_n} \le \bigO{n^{-\varepsilon + \delta}},
\end{equation}
as $n \to \infty$.

Now, observe that $p_n$ is the smallest degree such that nodes $i$ with degree $D_i > p_n$ will be 
connected to nodes with degree $1$, by the \texttt{DGA}, while $p$ is the corresponding degree for
the limit distribution. Therefore we have
\begin{equation}\label{eq:cases_psi}
	\psi(k,\ell) = \begin{cases}
		1 &\mbox{for all } k > p \text{ and } \ell = 1 \\
		1 &\mbox{for all } k = 1 \text{ and } \ell > p \\
		\psi(k,\ell) &\mbox{for all } k \le p \text{ and } \ell \le p \\
		0 &\mbox{else}
	\end{cases},
\end{equation}
while, on the event $P_n$, the same relations hold for $\psi_n$. The idea of the
proof is to split the analysis into the three regions 
\[
	(k = 1,\, \ell > p), \quad (k,\, \ell \le p) \quad \text{and} \quad (k > p,\, \ell = 1).
\]
The hard work is in the second region. However, since on the event $\Lambda_n$ all degree are
bounded by $p$, it suffices to analyze individual terms 
\[
	\left|\psi_n(k,\ell)\mathcal{E}_n(k,\ell) - \psi(k,\ell)\mathcal{E}(k,\ell)\right|,
\]
instead of the full sum
\[
	\sum_{k, \ell = 0}^p \left|\psi_n(k,\ell)\mathcal{E}_n(k,\ell) - \psi(k,\ell)\mathcal{E}(k,\ell)\right|.
\]

Recall that 
\[
	\Xi_n = \left\{\sum_{k, \ell = 0}^\infty \left|h_n(k,\ell) - h(k, \ell)\right| \le 
	K n^{-\delta}\right\}.
\]
and let us bound the probability in \eqref{eq:thm_joint_degrees_main} as follows,
\begin{align}
	\Prob{\Xi_n^c, \Lambda_n} &\le \Prob{\sum_{k,\ell = 1}^\infty \left|\psi_n(k,\ell) 
		- \psi(k,\ell)\right| \mathcal{E}_n(k,\ell) > \frac{K n^{-\delta}}{2}, \Lambda_n} 
		\label{eq:joint_degrees_bound1}\\
  &\hspace{10pt}+ \Prob{\sum_{k,\ell = 1}^\infty \psi(k,\ell)\left|
    (\mathcal{E}(k,\ell) - \mathcal{E}_n(k, \ell))\right| > \frac{K n^{-\delta}}{2}, \Lambda_n}.
		\label{eq:joint_degrees_bound2}
\end{align}

We will first deal with \eqref{eq:joint_degrees_bound1}. By \eqref{eq:cases_psi} and 
conditioned on $\Lambda_n$, we have
that $\left|\psi_n(k,\ell) - \psi(k,\ell)\right| \ne 0$, only when $k, \ell \le p$. Hence we get,
using the union bound,
\begin{align*}
	&\hspace{-30pt}\Prob{\sum_{k,\ell = 1}^\infty \left|\psi_n(k,\ell) - \psi(k,\ell)\right| 
		\mathcal{E}_n(k,\ell) > \frac{K n^{-\delta}}{2}, \Lambda_n} \\
	&= \Prob{\sum_{k,\ell = 1}^p \left|\psi_n(k,\ell) - \psi(k,\ell)\right| 
		\mathcal{E}_n(k,\ell) > \frac{K n^{-\delta}}{2}, \Lambda_n} \\
	&\le \sum_{k,\ell = 1}^p \Prob{\left|\psi_n(k,\ell) - \psi(k,\ell)\right| 
		\mathcal{E}_n(k,\ell) > \frac{K n^{-\delta}}{2p^2}, \Lambda_n} \\
	&\le \bigO{n^{-\varepsilon + \delta}},
\end{align*}
where the last line follows from Lemma \ref{lem:technical_psi_n_vs_psi}.

Next we consider \eqref{eq:joint_degrees_bound2}. First we use \eqref{eq:cases_psi} to 
bound the term inside the probability as follows
\begin{align}
	\sum_{k,\ell = 1}^\infty \psi(k,\ell)\left|
    	\left(\mathcal{E}(k,\ell) - \mathcal{E}_n(k,\ell)\right)\right| 
    &\le \sum_{k = 1}^{p} \sum_{\ell = 1}^{p} \psi(k,\ell) \left|
    	\mathcal{E}(k,\ell) - \mathcal{E}_n(k,\ell)\right| \label{eq:thm_joint_degree_k_ell} \\
    &\hspace{10pt}+ \sum_{\ell = p + 1}^\infty  \left|
    	\mathcal{E}(1,\ell) - \mathcal{E}_n(1,\ell)\right| \label{eq:thm_joint_degree_ell} \\
    &\hspace{10pt}+ \sum_{k = p + 1}^\infty  \left|
    	\mathcal{E}(k,1) - \mathcal{E}_n(k,1)\right| \label{eq:thm_joint_degree_k}
\end{align}
We will start by analyzing \eqref{eq:thm_joint_degree_ell}. For this we notice that 
$\mathcal{E}_n(1,\ell) - \mathcal{E}(1, \ell) = f_n^\ast(\ell) - f^\ast(\ell)$, so that
\begin{align*}
	\sum_{\ell = p + 1}^\infty \psi(1,\ell)\left|\mathcal{E}(1,\ell) - \mathcal{E}_n(1,\ell)\right| 
    &\le \sum_{\ell = 0}^\infty \left|f_n^\ast(\ell) - f^\ast(\ell)\right|.
\end{align*}
The upper bound for \eqref{eq:thm_joint_degree_k} is the same. Therefore, again using the union
bound, we have that
\begin{align*}
	&\Prob{\sum_{k,\ell = 0}^\infty \psi(k,\ell)\left|
    \left(\mathcal{E}(k,\ell) - \mathcal{E}_n(k,\ell)\right)\right| > \frac{K n^{-\delta}}{2}, \Lambda_n} \\
	&\le 2\Prob{\left|f_n^\ast(\ell) - f^\ast(\ell)\right| > \frac{K n^{-\delta}}{6}, \Omega_n} 
		+ \sum_{k, \ell = 0}^p \Prob{\left|\mathcal{E}(k,\ell) - \mathcal{E}_n(k,\ell)\right|
		> \frac{K n^{-\delta}}{6p^2}, \Omega_n} \\
	&\le \bigO{n^{-\varepsilon + \delta}}.
\end{align*}
Here we used Lemma \ref{lem:technical_mathcalE_n_vs_mathcalE} to bound the last probability in the second line.

With this final result we have proven \eqref{eq:thm_joint_degrees_main} and hence Theorem
\ref{thm:joint_degree_distribution}.
\end{proof}

All that is left is to prove the two technical lemmas \ref{lem:technical_psi_n_vs_psi} and 
\ref{lem:technical_mathcalE_n_vs_mathcalE}. Due to the use of both a minimum 
and maximum, in the definitions of $\mathcal{E}_n(k, \ell)$ and $\mathcal{E}(k, \ell)$ and the double cases
in $\psi_n(k, \ell)$ and $\psi(k,\ell)$, the proofs 
consists of many case distinctions, where we have to bound each specific case. In order to improve
the readability of the proofs we define, for any $k, \ell \ge 0$, the following events
\begin{align*}
	A_n &= \left\{1 - F_n^\ast(k) < F_n^\ast(\ell)\right\} \\
    B_n &= \left\{1 - F_n^\ast(k - 1) > F_n^\ast(\ell - 1)\right\}, \\
    I_n &= \left\{1 - F_n^\ast(k - 1) \le F_n^\ast(\ell)\right\}, \\
    J_n &= \left\{1 - F_n^\ast(k) \ge F_n^\ast(\ell - 1)\right\}.
\end{align*}
With these definitions we have that $\psi_n(k,\ell) = \ind{A_n}\ind{B_n}$. Moreover since $A^I_n 
\cap B_n^c = \emptyset$ we have that
\begin{equation}\label{eq:complemnt_psi_n}
	1 - \psi_n(k,\ell) = \ind{A_n}\ind{B_n^c} + \ind{A_n^c}\ind{B_n}.
\end{equation}
Where the event $A_n$ and $B_n$ determine the value of $\psi_n(k, \ell)$, so do the events $I_n$ 
and $J_n$ define the expression for $\mathcal{E}_n(k,\ell)$, as follows:
\begin{equation}\label{eq:cases_mathcalE_n}
	\mathcal{E}_n(k, \ell) = \begin{cases}
    	f_n^\ast(k) &\mbox{on the event } I_n \cap J_n \\
        1 - F^\ast(k - 1) - F^\ast(\ell - 1) &\mbox{on the event } I_n \cap J_n^c \\
        F^\ast(k) + F^\ast(\ell) - 1 &\mbox{on the event } I_n^c \cap J_n \\
        f_n^\ast(\ell) &\mbox{on the event } I_n^c \cap J_n^c.
    \end{cases}
\end{equation}
Note that by their definitions,
\[
	0 \le \, \psi_n(k, \ell),\, \psi(k, \ell),\, \mathcal{E}_n(k,\ell),\, \mathcal{E}(k,\ell)\, \le 1,
\]
for all $k, \ell \ge 0$. In addition we will often use the following result

\begin{lemma}\label{lem:technical_inequalities_Fast}
Let $k, \ell \ge 0$ be such that $1 - F^\ast(k) < F^\ast(\ell)$. Then
\[
	\Prob{1 - F_n^\ast(k) \ge F^\ast_n(\ell), \Omega_n} \le \bigO{n^{-\varepsilon}},
\]
as $n \to \infty$. 

\noindent If, on the other hand, $1 - F^\ast(k) > F^\ast(\ell)$, then
\[
	\Prob{1 - F_n^\ast(k) \le F^\ast_n(\ell), \Omega_n} \le \bigO{n^{-\varepsilon}},
\]
as $n \to \infty$.
\end{lemma}

\begin{proof}
We will prove the first statement, since the proof for the second is similar.
First we write
\begin{align*}
	&\Prob{1 - F_n^\ast(k) \ge F^\ast_n(\ell), \Omega_n} \\
	&\Prob{\left(F^\ast(k) - F_n^\ast(k)\right) + 1 - F^\ast(k) 
		\ge F^\ast(\ell) + \left(F^\ast_n(\ell) - F^\ast(\ell)\right), \Omega_n} \\
	&\Prob{\left(F^\ast(k) - F_n^\ast(k)\right) + \left(F^\ast_n(\ell) - F^\ast(\ell)\right)  
		\ge F^\ast(\ell) - 1 + F^\ast(k), \Omega_n}.
\end{align*}
Next we use the union bound and Markov's inequality to obtain
\begin{align*}
	&\Prob{1 - F_n^\ast(k) \ge F^\ast_n(\ell), \Omega_n} \\
	&\le \Prob{\left|F^\ast(k) - F_n^\ast(k)\right| 
		\ge F^\ast(\ell) - 1 + F^\ast(k), \Omega_n} \\
	&\hspace{10pt}+ \Prob{\left|F^\ast(\ell) - F_n^\ast(\ell)\right| 
		\ge F^\ast(\ell) - 1 + F^\ast(k), \Omega_n} \\
	&\le \frac{2\Exp{\sup_{k \ge 0}|F_n^\ast(k) - F^\ast(k)|\ind{\Omega_n}}}
		{F^\ast(\ell) - 1 + F^\ast(k)} \\
	&\le \Exp{\sum_{k = 0}^\infty|f_n^\ast(k) - f^\ast(k)|\ind{\Omega_n}} 
		= \bigO{n^{-\varepsilon}},
\end{align*}
as $n \to \infty$, where we used $1 - F^\ast(k) < F^\ast(\ell)$ for the last equality.
\end{proof}

\begin{proof}[Proof of Lemma \ref{lem:technical_psi_n_vs_psi}]
Note that the specific expression of $\psi(k,\ell)$ depends on the ordering between 
\[
	1 - F^\ast(k) \quad \text{and} \quad F^\ast(\ell),
\]
and 
\[
	1 - F^\ast(k - 1) \quad \text{and} \quad F^\ast(\ell - 1).
\]
Therefore, we need to consider all different cases (<, =, >), where we treat equality as a separate case. 
This gives a total of nine cases. However, there are several combinations that do not need to be 
considered. For instance, $1 - F^\ast(k) > F^\ast(\ell)$ implies that $1 - F^\ast(k - 1) \ge 
F^\ast(\ell - 1)$. In the end, we are left with the following cases:
\begin{enumerate}[\upshape I)]
	\item $1 - F^\ast(k) < F^\ast(\ell)$ and $1 - F^\ast(k - 1) < F^\ast(\ell - 1)$ 
	\item $1 - F^\ast(k) = F^\ast(\ell)$ and $1 - F^\ast(k - 1) < F^\ast(\ell - 1)$ 
    \item $1 - F^\ast(k) < F^\ast(\ell)$ and $1 - F^\ast(k - 1) = F^\ast(\ell - 1)$ 
    \item $1 - F^\ast(k) < F^\ast(\ell)$ and $1 - F^\ast(k - 1) > F^\ast(\ell - 1)$ 
    \item $1 - F^\ast(k) = F^\ast(\ell)$ and $1 - F^\ast(k - 1) > F^\ast(\ell - 1)$ 
    \item $1 - F^\ast(k) > F^\ast(\ell)$ and $1 - F^\ast(k - 1) > F^\ast(\ell - 1)$ 
\end{enumerate}
We will start with the first case.
\par \smallskip

\noindent \textbf{I)} $\bm{1 - F^\ast(k) < F^\ast(\ell)}$ \textbf{and} $\bm{1 - F^\ast(k - 1) 
< F^\ast(\ell - 1)}$ \par \smallskip

First, note that in this case $\psi(k,\ell) = 0$. Moreover, since $F^\ast(\ell - 1) > 
1 - F^\ast(k - 1)$, it follows from Lemma \ref{lem:technical_inequalities_Fast} that
\begin{align*}
	\Prob{B_n} &\le \Prob{B_n, \Omega_n} + \Prob{\Omega_n^c} \\
    &\le O\left(n^{-\varepsilon} + \Prob{\Omega_n^c}\right).
\end{align*}
Hence, since $\psi_n(k,\ell) = 0$ on the event $B_n^c$, we have
\begin{align*}
	&\hspace{-30pt}\Prob{\left|\psi_n(k,\ell) - \psi(k,\ell)\right| \mathcal{E}_n(k,\ell) 
		> Kn^{-\delta}} \\
  &= \Prob{\psi_n(k,\ell) \mathcal{E}_n(k,\ell) > Kn^{-\delta}, B_n^c} + \Prob{B_n} \\
  &\le O\left(n^{-\varepsilon} + \Prob{\Omega_n^c}\right).
\end{align*}

\noindent \textbf{II)} $\bm{1 - F^\ast(k) = F^\ast(\ell)}$ \textbf{and} $\bm{1 - F^\ast(k - 1) 
< F^\ast(\ell - 1)}$ \par \smallskip

In this case we again have that $\psi(k,\ell) = 0$. In addition 
\[
	1 - F^\ast(k - 1) > 1 - F^\ast(k) > F^\ast(\ell),
\]
so that, by Lemma \ref{lem:technical_inequalities_Fast}
\begin{align*}
	\Prob{I_n} &\le O\left(n^{-\varepsilon} + \Prob{\Omega_n^c}\right).
\end{align*}
Similarly, using that $1 - F^\ast(k) > F^\ast(\ell - 1)$, we have
\[
	\Prob{J_n^c} \le O\left(n^{-\varepsilon} + \Prob{\Omega_n^c}\right).
\]
Therefore, using \eqref{eq:cases_mathcalE_n} and $1 - F^\ast(k) = F^\ast(\ell)$, it follows that
\begin{align*}
	&\hspace{-30pt}\Prob{\left|\psi_n(k,\ell) - \psi(k,\ell)\right| \mathcal{E}_n(k,\ell) > 
		K n^{-\delta}}\\
	&= \Prob{\psi_n(k,\ell) \mathcal{E}_n(k,\ell) > K n^{-\delta}} \\
	&\le \Prob{\mathcal{E}_n(k,\ell) > Kn^{-\delta}, I_n^c, J_n} + \Prob{I_n} + \Prob{J_n^c} 
		+ \Prob{I_n, J_n^c}\\
	&\le \Prob{\left|F_n^\ast(\ell) + F_n^\ast(k) - 1\right| > K n^{-\delta}} + O\left(n^{-\varepsilon}
		+ \Prob{\Omega_n^c}\right)\\
	&\le \Prob{\left|F_n^\ast(\ell) - F^\ast(\ell)\right| > \frac{K n^{-\delta}}{2}} \\
	&\hspace{10pt}+ \Prob{\left|F_n^\ast(k) - F^\ast(k)\right| > \frac{K n^{-\delta}}{2}} 
			+ O\left(n^{-\varepsilon} + \Prob{\Omega_n^c}\right)\\
	&\le + O\left(n^{-\varepsilon + \delta} + \Prob{\Omega_n^c}\right),
\end{align*}
where for the fifth line we used that
\begin{align*}
	\left|F_n^\ast(\ell) + F_n^\ast(k) - 1\right| 
    &= \left|F_n^\ast(\ell) - F^\ast(\ell) + (1 - F^\ast(k)) + F_n^\ast(k) - 1\right| \\
    &\le \left|F_n^\ast(\ell) - F^\ast(\ell)\right| + \left|F_n^\ast(k) - F^\ast(k)\right|,
\end{align*}
since $1 - F^\ast(k) = F^\ast(\ell)$.

Case III) and V) can be dealt with using arguments similar to case II), while case VI) is similar 
to I). Therefore, there is only one case left. \par \smallskip

\noindent \textbf{IV)} $\bm{1 - F^\ast(k) < F^\ast(\ell)}$ \textbf{and} $\bm{1 - F^\ast(k - 1) 
	> F^\ast(\ell - 1)}$ \par \smallskip

We first note that, since $1 - F^\ast(k) - F^\ast(\ell) > 0$,
\begin{align*}
	\Prob{A_n^c} 
	&\le O\left(n^{-\varepsilon} + \Prob{\Omega_n^c}\right),
\end{align*}
by Lemma \ref{lem:technical_inequalities_Fast}, and similarly
\[
	\Prob{B_n^c} \le O\left(n^{-\varepsilon} + \Prob{\Omega_n^c}\right).
\]
Since for this case $\psi(k,\ell) = 1$, we have,
\begin{align*}
	&\hspace{-30pt}\Prob{\left|\psi_n(k,\ell) - \psi(k,\ell)\right| \mathcal{E}_n(k,\ell) 
		> K n^{-\delta}} \\
	&= \Prob{(1 - \psi_n(k,\ell)) \mathcal{E}_n(k,\ell) > K n^{-\delta}}\\
	&\le \Prob{(1 - \psi_n(k,\ell)) \mathcal{E}_n(k,\ell) > K n^{-\delta}, A_n, B_n} \\
	&\le \Prob{A_n^c} + \Prob{B_n^c} + \Prob{A_n^c, B_n^c} \\
	&\le O\left(n^{-\varepsilon} + \Prob{\Omega_n^c}\right),
\end{align*}
where we used \eqref{eq:complemnt_psi_n} for the third line.

\end{proof}

\begin{proof}[Proof of Lemma \ref{lem:technical_mathcalE_n_vs_mathcalE}]
Similar to the proof of Lemma \ref{lem:technical_psi_n_vs_psi} we will have to consider 
different cases. Here these are with respect to the different relations between
\[
	1 - F^\ast(k - 1) \quad \text{and} \quad F^\ast(\ell),
\]
and 
\[
	1 - F^\ast(k) \quad \text{and} \quad F^\ast(\ell - 1),
\]
which determine the expression for $\mathcal{E}(k,\ell)$. To analyze each case we will also need to
distinguish between the different expression of $\mathcal{E}_n(k,\ell)$, which are determined by 
the events $I_n$ and $J_n$. \par \smallskip

We will consider the three cases where $1 - F^\ast(k) > F^\ast(\ell - 1)$. The other six cases can 
be dealt with using similar arguments. First note that by Lemma 
\ref{lem:technical_inequalities_Fast}
\begin{align*}
	\Prob{J_n^c} \
  &\le O\left(n^{-\varepsilon} + \Prob{\Omega_n^c}\right).
\end{align*}

\noindent \textbf{I)} $\bm{1 - F^\ast(k - 1) < F^\ast(\ell)}$ \textbf{and} $\bm{1 - F^\ast(k) > 
F^\ast(\ell - 1)}$ \par \smallskip

Similar to $\Prob{J_n^c}$, it follows from Lemma \ref{lem:technical_inequalities_Fast} that
\begin{align*}
	\Prob{I_n^c}
  &\le O\left(n^{-\varepsilon} + \Prob{\Omega_n^c}\right).
\end{align*}
Therefore, by conditioning on the different combinations of $I_n$ and $J_n$, we get
\begin{align*}
	&\hspace{-20pt}\Prob{\left|\mathcal{E}(k,\ell) - \mathcal{E}_n(k,\ell)\right| > Kn^{-\delta}} \\
    &\le \Prob{\left|\mathcal{E}(k,\ell) - \mathcal{E}_n(k,\ell)\right| > Kn^{-\delta}, I_n, J_n}
    	+ \Prob{J_n^c} + \Prob{I_n^c} + \Prob{J_n^c, I_n^c} \\
    &\le \Prob{\left|f^\ast(k) - f_n^\ast(k)\right| > Kn^{-\delta}, \Omega_n} 
    	+ O\left(n^{-\varepsilon} + \Prob{\Omega_n^c}\right) \\
    &\le O\left(n^{-\varepsilon + \delta} + \Prob{\Omega_n^c}\right).
\end{align*}

\noindent \textbf{II)} $\bm{1 - F^\ast(k - 1) = F^\ast(\ell)}$ \textbf{and} $\bm{1 - F^\ast(k) > 
F^\ast(\ell - 1)}$ \par \smallskip

Since $f^\ast(k) = F^\ast(k) - F^\ast(k-1)$,
\begin{align*}
	\left|F_n^\ast(k) + F_n^\ast(\ell) - 1 - f^\ast(k)\right| 
    &= \left|F_n^\ast(k) - F^\ast(k) + F_n^\ast(\ell) - 1 + F^\ast(k - 1)\right| \\
    &\le \left|F_n^\ast(k) - F^\ast(k)\right| + \left|F_n^\ast(\ell) - F^\ast(\ell)\right|,
\end{align*}
from which it follows that
\begin{align*}
	&\hspace{-30pt}\Prob{\left|\mathcal{E}(k,\ell) - \mathcal{E}_n(k,\ell)\right| > Kn^{-\delta}, 
		I_n^c, J_n} \\
	&\le \Prob{\left|F_n^\ast(k) + F_n^\ast(\ell) - 1 - f^\ast(k)\right| > Kn^{-\delta}} \\
	&\le \Prob{\left|F_n^\ast(k) - F^\ast(k)\right| > \frac{Kn^{-\delta}}{2}}
		+ \Prob{\left|F_n^\ast(\ell) - F^\ast(\ell)\right| > \frac{Kn^{-\delta}}{2}} \\
	&\le O\left(n^{-\varepsilon + \delta} + \Prob{\Omega_n^c}\right).
\end{align*}
Hence, we obtain
\begin{align*}
	&\hspace{-30pt}\Prob{\left|\mathcal{E}(k,\ell) - \mathcal{E}_n(k,\ell)\right| > Kn^{-\delta}} \\
	&\le \Prob{\left|f^\ast(k) - f_n^\ast(k)\right| > Kn^{-\delta}} \\
	&\hspace{10pt}+ \Prob{\left|\mathcal{E}(k,\ell) - \mathcal{E}_n(k,\ell)\right| > Kn^{-\delta}, 
		I_n^c, J_n}	+ 2\Prob{J_n^c}\\
	&\le O\left(n^{-\varepsilon + \delta} + \Prob{\Omega_n^c}\right).
\end{align*}

\noindent \textbf{III)} $\bm{1 - F^\ast(k - 1) > F^\ast(\ell)}$ \textbf{and} $\bm{1 - F^\ast(k) > 
F^\ast(\ell - 1)}$ \par \smallskip

First we notice that in this case $\mathcal{E}(k,\ell) = F^\ast(\ell) + F^\ast(k) - 1$. Next, 
using Lemma \ref{lem:technical_inequalities_Fast}, we have
\[
	\Prob{I_n} \le O\left(n^{-\varepsilon + \delta} + \Prob{\Omega_n}\right).
\]
Therefore it follows that
\begin{align*}
	&\hspace{-30pt}\Prob{\left|\mathcal{E}(k,\ell) - \mathcal{E}_n(k,\ell)\right| > Kn^{-\delta}} \\
	&\le \Prob{\left|F^\ast(\ell) + F^\ast(k) - 1 - \mathcal{E}_n(k,\ell)\right| > Kn^{-\delta}, 
		I_n^c, J_n} \\
	&\hspace{10pt}+ \Prob{I_n} + 2\Prob{J_n^c} \\
	&\le \Prob{\left|F_n^\ast(k) - F^\ast(k)\right| > \frac{Kn^{-\delta}}{2}}
		+ \Prob{\left|F_n^\ast(\ell) - F^\ast(\ell)\right| > \frac{Kn^{-\delta}}{2}} \\
	&\hspace{10pt}+ \Prob{I_n} + 2\Prob{J_n^c}\\
	&\le O\left(n^{-\varepsilon + \delta} + \Prob{\Omega_n^c}\right).
\end{align*}
\end{proof}

\subsection{Main results}\label{ssec:proof_main_result}

Here we will give the proofs of our two main results. We start with a useful result which we need to prove Theorem \ref{thm:main_result}. 

\begin{proposition}\label{prop:convergence_spearman}
Let $G_n \in \mathcal{G}_{\eta,\,\varepsilon}(f, f^\ast)$ and let $X, Y$ be random variables with joint distribution $h$ as defined in \eqref{eq:joint_degree_density}.
Then, for any
$0 < \delta < \varepsilon$ and $K > 0$, 
\[
	\Prob{\left|\widetilde{\rho}(G_n) - \rho(D_\ast, D^\ast)\right| > n^{-\delta}} = \bigO{n^{-\varepsilon + \delta} + \Prob{\Omega_n^c}},
\]
as $n \to \infty$.
\end{proposition}

\begin{proof}
First we write
\begin{align*}
	\left|\widetilde{\rho}(\widehat{G}_n) - \rho(X,Y)\right| 
    &\le 3\left|\sum_{k, \ell = 0}^\infty \mathcal{F}_n^\ast(k)
    	\mathcal{F}_n^\ast(\ell) h_n(k, \ell) - \mathcal{F}^\ast(k)
        \mathcal{F}^\ast(\ell) h(k, \ell)\right|\\
    &\le 3\left|\sum_{k, \ell = 0}^\infty \mathcal{F}_n^\ast(k)
    	\mathcal{F}_n^\ast(\ell) \left(h_n(k, \ell) - h(k, \ell)\right)\right| \\
    &\hspace{10pt}+ 3\sum_{k, \ell = 0}^\infty \left|\mathcal{F}_n^\ast(k)
    	\mathcal{F}_n^\ast(\ell) - \mathcal{F}^\ast(k)\mathcal{F}^\ast(\ell) \right|h_n(k, \ell)
    	\\
    &\le 12 \sum_{k, \ell = 0}^\infty \left|h_n(k, \ell) - h(k, \ell)\right|
    	+ 24 \sup_{k} \left|F_n^\ast(k) - F^\ast(k)\right|\numberthis 
        \label{eq:main_result_upper_bound}.
\end{align*}
For the last inequality, we used
\begin{align*}
	&\hspace{-40pt}\sum_{k, \ell = 0}^\infty \left|\mathcal{F}_n^\ast(k)
    	\mathcal{F}_n^\ast(\ell) - \mathcal{F}^\ast(k)\mathcal{F}^\ast(\ell) \right|h_n(k, \ell) \\
    &\le \sup_{k, \ell} \left|\mathcal{F}_n^\ast(k)\mathcal{F}_n^\ast(\ell) 
    	- \mathcal{F}^\ast(k)\mathcal{F}^\ast(\ell) \right| \\
    &\le \sup_{k, \ell} \left|\mathcal{F}_n^\ast(k) - \mathcal{F}^\ast(k)\right| \mathcal{F}_n^
    	\ast(\ell) + \sup_{k, \ell} \left|\mathcal{F}_n^\ast(\ell) - \mathcal{F}^\ast(\ell)\right| 
        \mathcal{F}^\ast(k) \\
    &\le 4 \sup_{k} \left|\mathcal{F}_n^\ast(k) - \mathcal{F}^\ast(k)\right|
    \le 8 \sup_{k} \left|F_n^\ast(k) - F^\ast(k)\right|.
\end{align*}
Note that by Theorem \ref{thm:joint_degree_distribution}
\[
	\Prob{12 \sum_{k, \ell = 0}^\infty \left|h_n(k, \ell) - h(k, \ell)\right| > \frac{n^{-\delta}}{2}}
	= \bigO{n^{-\varepsilon + \delta} + \Prob{\Omega_n^c}}.
\]
Moreover, on the event $\Omega_n$,
\[
	\sup_{k \ge 0} \left|F_n^\ast(k) - F^\ast(k)\right| \le \sum_{k = 0}^\infty |f_n^\ast(k) - f^\ast(k)| \le n^{-\varepsilon}.
\]
Hence, it follows from \eqref{eq:main_result_upper_bound} and Markov's inequality that
\begin{align*}
	\Prob{\left|\rho(\tilde{G}_n - \rho(X,Y)\right| > n^{-\delta}} 
    &\le \Prob{12 \sum_{k, \ell = 0}^\infty \left|h_n(k, \ell) - h(k, \ell)\right| 
    	> \frac{n^{-\delta}}{2}} \\
    &\hspace{10pt}+ \Prob{24 \sup_{k} \left|F_n^\ast(k) - F^\ast(k)\right| > 
        \frac{n^{-\delta}}{2}, \Omega_n} + O\left(\Prob{\Omega_n^c}\right)\\
    &\le O\left(n^{-\varepsilon + \delta} + \Prob{\Omega_n^c}\right) 
    	+ 48 n^{\delta} \Exp{\sup_{k} \left|F_n^\ast(k) - F^\ast(k)\right|\ind{\Omega_n}} \\
    &\le O\left(n^{-\varepsilon + \delta} + \Prob{\Omega_n^c}\right).
\end{align*}
\end{proof}

We are now ready to give the proof of Theorem \ref{thm:main_result}.

\begin{proof}[Proof of Theorem \ref{thm:main_result}]
Consider a graph $G_n \in \mathcal{G}_{\eta, \varepsilon}(f,f^\ast)$, denote its degree sequence by ${\bf D}_n$,
let $\widetilde{G}_n = \texttt{DGA}({\bf D}_n)$ and recall that $\kappa = (\varepsilon + \delta)/2$.
Then, since $\delta < \kappa < \varepsilon$, it follows from Proposition \ref{prop:convergence_spearman} that
\[
	\Prob{\left|\widetilde{\rho}(\widehat{G}_n) - \rho(D_\ast,D^\ast)\right| > K n^{-\delta}} 
	\le \bigO{n^{-\varepsilon + \kappa} + \Prob{\Omega_n^c}},
\]
which proves the second statement of the theorem.

For the first statement, note that by Theorem \ref{thm:optimality_algorithm} 
\[
	\sum_{i \to j \in G_n} \mathcal{F}^\ast_n(D_i)\mathcal{F}^\ast_n(D_j) \ge
	\sum_{i \to j \in \widetilde{G}_n} \mathcal{F}^\ast_n(D_i)\mathcal{F}^\ast_n(D_j)
\]
so that
\[
	\widetilde{\rho}(G_n) \ge \widetilde{\rho}(\widehat{G}_n).
\]
Therefore we have, as $n \to \infty$,
\begin{align*}
	\Prob{\widetilde{\rho}(G_n) < \rho(D_\ast, D^\ast) - K n^{-\delta}}
	&\le \Prob{\widetilde{\rho}(\widehat{G}_n) < \rho(D_\ast, D^\ast) - n^{-\delta}}\\
	&\le \Prob{\left|\widetilde{\rho}(\widetilde{G}_n) - \rho(D_\ast, D^\ast) \right| 
		> K n^{-\delta}} \\
	&\le \bigO{n^{-\varepsilon - \kappa} + \Prob{\Omega_n^c}},
\end{align*}
which proves the first statement of the theorem.
\end{proof}

We now move on to Theorem \ref{thm:independence_spearman_tail}. We will follow the strategy described in Section \ref{sec:spearman_tail_distribution}, that is we will use the delta transformation to construct a degree distribution $f_\rho$ for which $f_\rho^\ast(1)$ is large enough.

First observe that \eqref{eq:spearman_lower_bound_a1} together with Proposition \ref{prop:convergence_spearman} imply \ref{prop:lower_bound_spearman}.

\begin{proof}[Proof of Theorem \ref{thm:independence_spearman_tail}]
Let $\delta$ be such that $9(\xi/2)^2 - 6(\xi/2)^3 - 3 = \rho + \epsilon$, for some $\epsilon > 0$, and 
denote by $f^\ast$ the size-biased distribution of $f$. Now take $f^\ast_\delta$ to be the 
$\delta$-transform of $f^\ast$ and set
\begin{align*}
    \mu_\rho = \left(\mu(1 - F(K_\delta)) + \sum_{t = 1}^{K_\delta} \frac{f^\ast_\delta(t)}{t}\right)^{-1},
\end{align*}
where $K_\delta$ was defined as the smallest integer such that $F^\ast(K_\delta) > \delta$. 
Now we define the function $f_\rho$ by:
\begin{align*}
	f_\rho(0) = \frac{\mu_\rho f^\ast_\delta(1)}{2} = f_\rho(1) &= \frac{\mu_\rho f^\ast_\delta(1)}{2} \quad \text{and} \quad
	f_\rho(t) = \frac{\mu_\rho f^\ast_\delta(t)}{t} \quad \text{for } k \ge 2.
\end{align*}
Then, since by construction $f_\rho^\ast(t) = f^\ast(t)$ for all $t > K_\delta$, it follows that 
\begin{align*}
	\sum_{t = 0}^\infty f_\rho(t) &= \sum_{t = 1}^\infty \frac{\mu_\rho f_\rho^\ast(t)}{t} \\
    &= \mu_\rho\left(\sum_{t = 1}^{K_\delta} \frac{f_\rho^\ast(t)}{t} + \sum_{t = K_\delta + 1}^\infty \frac{f^\ast(t)}{t}\right) \\
    &= \mu_\rho\left(\sum_{t = 1}^{K_\delta} \frac{f_\rho^\ast(t)}{t} + \mu(1 - F(K_\delta))\right) = 1,
\end{align*}
so that $f_\rho$ defines a probability density. Moreover, since for all $k > K_\delta$
\[
	1 - F_\rho(k) = \sum_{t = k + 1}^\infty f_\rho(t) = \mu_\rho \sum_{t = k + 1}^\infty \frac{f_\delta^\ast(t)}{t}
    = \mu_\rho \sum_{t = k + 1}^\infty \frac{f^\ast(t)}{t} = \frac{\mu_\rho}{\mu} \sum_{t = k + 1}^\infty f(t),
\]
it follows that $\sum_{k = 0}^\infty t^{1 + \eta}f_\rho(t) < \infty$ and
\[
	\lim_{k \to \infty} \frac{1 - F_\rho(k)}{1 - F(k)} = \frac{\mu_\rho}{\mu}.
\]

Now let $D$ have probability density $f_\rho$, and hence size-biased density 
$f_\rho^\ast(t) = tf_\rho(t)/\mu_\rho$, and let ${\bf D}_n$ be generated by the \texttt{IID} algorithm, by sampling from $D$. Then, by 
Lemma \ref{lem:iid_algorithm}, ${\bf D}_n \in \mathcal{D}_{\eta, \varepsilon}(
f_\rho, f_\rho^\ast)$ and since by construction of $f_\rho$ we have that $f_\rho^\ast(1) = \delta/2$, 
it follows that
\[
	9f_\rho^\ast(1)^2 - 6f_\rho^\ast(1)^3 - 3 = \rho + \epsilon.
\]
Hence, if $G_n$ is a graph with degree sequence ${\bf D}_n$, we have, by taking $\delta = 
\min(\varepsilon, 1/2)/2$ in Proposition \ref{prop:lower_bound_spearman}, that as $n \to 
\infty$,
\begin{align*}
	\Prob{\rho(G_n) > \rho} &\ge \lim_{n \to \infty}\Prob{\rho(G_n) > \rho + \epsilon 
		- n^{-\delta}} \\
  &= \Prob{\rho(G_n) > 9f_\rho^\ast(1)^2 - 6f_\rho^\ast(1)^3 - 3 - n^{-\delta}} \\
	&\ge 1 - \bigO{n^{-\varepsilon + 3\kappa/4} + \Prob{\Omega_n^c}}.
\end{align*}
\end{proof}

\bibliographystyle{plain}
\bibliography{references} 

\end{document}